\documentclass[11pt,a4paper,twoside]{amsart}

\usepackage{amsmath,amsfonts,amssymb,amsthm}
\usepackage{dsfont}	
\usepackage[T1]{fontenc}

\usepackage[disable]{todonotes} 

\usepackage{verbatim}
\usepackage[english]{babel}
\usepackage{a4wide}
\usepackage[small,euler-digits,icomma,OT1,T1]{eulervm}
\usepackage{hyperref}
\usepackage{color}

\DeclareMathOperator{\Prb}{\mathbf{P}}
\DeclareMathOperator{\Mean}{\mathbf{E}}

\DeclareMathOperator{\Law}{Law}

\newcommand{\R}{\mathbb{R}}

\theoremstyle{plain}
\newtheorem{theorem}{Theorem}[section]
\newtheorem{proposition}[theorem]{Proposition}
\newtheorem{lemma}[theorem]{Lemma}
\newtheorem{corollary}[theorem]{Corollary}
\theoremstyle{definition}
\newtheorem{definition}{Definition}[section]
\theoremstyle{remark}
\newtheorem{remark}{Remark}[section]
\newtheorem{assumpt}{Assumption}

\title{McKean-Vlasov limit for interacting systems with simultaneous jumps}

\author{Luisa Andreis, Paolo Dai Pra, Markus Fischer}
\address{Dipartimento di Matematica \\ 
Universit\`a degli Studi di Padova\\ 
Via Trieste 63\\ 
35121 Padova, Italy}
\email{andreis@math.unipd.it, daipra@math.unipd.it, fischer@math.unipd.it}



\date{\today}

\begin{document}

\begin{abstract}
Motivated by several applications, including neuronal models, we consider the McKean-Vlasov limit for mean-field systems of interacting diffusions with simultaneous jumps. We prove propagation of chaos via a coupling technique that involves an intermediate process and that gives a rate of convergence for the $W_1$ Wasserstein distance between the empirical measures of the two systems on the space of trajectories $\mathbf{D}([0,T],\mathbb{R}^d)$. 
\end{abstract}

\maketitle

\section{Introduction}
Treatable modeling for complex systems often involves the {\em mean-field} assumption: the system is comprised by several interacting components, whose distribution is permutation invariant. This assumption allows in several cases the derivation of {\em  macroscopic} equations for the dynamics ({\em McKean-Vlasov} equations), in the limit as the number of components tends to infinity. Macroscopic behavior is also related to the phenomenon of {\em propagation of chaos}, which states that different components become stochastically independent as their number increases to infinity.

Since the introduction of this topic in the study of fluid dynamics (\cite{Kac54, McK66}), dynamic mean-field models have been considered both in general (\cite{Da83,Gra92,Gra92non, Szn91})  and for special models, motivated by life sciences (\cite{BaFaFaTo12, BeGiPa10, RoTo14, bo14}) and social sciences (\cite{drst, gar13, gie15}). General results on propagation of chaos include diffusions with jumps (\cite{Gra92, Gra92non}), under suitable Lipschitz conditions on coefficients. In recent years, neuronal networks have motivated the introduction of models whose components are allowed to jump simultaneously, and are therefore not covered by the results in \cite{Gra92}. Propagation of chaos is not obvious in these models, since simultaneous jumps could in principle interfere with asymptotic independence. In \cite{DeGaLoPr14}, this problem is solved via a rather involved approximation techniques, while existing techniques have been adapted to this context in \cite{FoLo14} and \cite{RoTo14}.

The purpose of this paper is to prove propagation of chaos results for general models with simultaneous jumps, which include the great majority of those cited above, by applying coupling techniques mainly borrowed from \cite{Gra92}. The peculiarity of this approach is the $L^1$ framework, reflected in the use of the Wasserstein-one distance. The $L^1$ framework, as opposed to the more common $L^2$ framework, allows for prescribing, in a natural way, the transition rates of the jumps. Besides simultaneous jumps, the models we treat may have drift or jump coefficients which are, though only locally Lipschitz as functions of the state variable, sufficiently stabilizing to yield strong existence and uniqueness of solutions. Notice that, when dealing with nonlinear Markov processes, the localization procedure usually employed in the proof of existence and uniqueness for \mbox{SDEs} with non-globally Lipschitz coefficients does not work in general \cite{scheu87}.
We combine coupling arguments with the introduction of an ``intermediate process'' that will be convenient in handling the jump terms. Thus, for many systems considered in the literature, no {\em ad hoc} technique is necessary to obtain propagation of chaos results. It also provides a framework that can be applied to generalizations of the models recently proposed.

The paper is organized as follows. In Section~\ref{Particle_system_and_limit}, we introduce at an informal level the framework, describing the main characteristics of both the particle system and the nonlinear process that we are interested in. In the following sections, we prove our results under precise assumptions, for three classes of systems. In Section~\ref{Lipschitz_cond}, we work under Lipschitz conditions, where rather straightforward adaptations of standard techniques are applied. In Section~\ref{non_lip_drift}, we introduce a class of nonlinear diffusions with jumps where the drift term comes from a convex potential. In Section~\ref{non_lip_rate}, we take inspiration from the neuroscience models mentioned above and adapt our techniques to a class of piecewise deterministic processes, where the jump rate is superlinear.

\section{Interacting particle systems and macroscopic limits}\label{Particle_system_and_limit}

In this Section, we introduce both the microscopic and the macroscopic dynamics at an informal level, and illustrate the phenomenon of propagation of chaos. In the remaining part of the paper well-posedness and convergence will be shown under various assumptions.

\subsection{The microscopic dynamics}

Let  $X^N=(X^N_1,\dots,X^N_N)$ $\in$ $\mathbb{R}^{d\times N}$ be the spatial positions of $N$ different particles moving in $\R^d$. We introduce the corresponding {\em empirical measure}
\begin{equation*}
	\mu^N_X \doteq \frac{1}{N}\sum_{i=1}^N\delta_{X^N_i}.
\end{equation*}
When the time variable appears explicitly in $X^N(t)$, we write $\mu^N_X(t)$ to indicate the time dependence of the empirical measure. Note that $\mu^N_X(t)$ is an element of $\mathcal{M}(\R^d)$, the set of probability measures on the Borel subsets of $\R^d$. 

The particle positions $X^N(t)$ evolve as a jump diffusion process with the following specifications for the $i$-th particle:
\begin{itemize}

\item a drift coefficient of the form $F(X^N_i(t),\mu_X^N(t))$ for some function $F: \R^d \times {\mathcal{M}}(\R^d) \rightarrow \R^d$ common to all particles;

\item a diffusion coefficient of the form $\sigma(X^N_i(t),\mu_X^N(t))$ for some function $\sigma: \R^d \times {\mathcal{M}}(\R^d) \rightarrow \R^{d\times d_1}$, again the same for all particles;

\item the jump amplitude and rate: particle $i$ jumps by a random amplitude  \\ $\psi(X^N_i(t),\mu_X^N(t),h^N_i)$ with rate $\lambda(X^N_i(t),\mu_X^N(t))$; this {\em main} jump induces simultaneous {\em collateral} jumps of {\em all} other particles: the $j$-th particle jumps by a random amplitude  $\displaystyle{\frac{\Theta(X^N_i(t),X^N_j(t),\mu_X^N(t),h^N_i,h^N_j)}{N}}$, where randomness of the jumps is given by the random parameter $h^N=(h^N_i)_{i=1,\dots,n}$ that is distributed according to a symmetric probability measure $\nu_N$ on $[0,1]^N$. Here, $\lambda$, $\Psi$, $\Theta$ are functions $\R^d \times {\mathcal{M}}(\R^d) \rightarrow [0,\infty)$, $\R^d \times {\mathcal{M}}(\R^d)\times [0,1] \rightarrow \R^d$, and $\R^d \times \R^d\times {\mathcal{M}}(\R^d)\times [0,1]^2 \rightarrow \R^d$, respectively.  

\end{itemize}

In more analytic terms, we are considering a  Markov process $X^N=\{X^N(t)\}_{t\in[0,T]}$ with values in $\mathbb{R}^{d\times N}$ whose infinitesimal generator takes the following form on a suitable family of test functions $f$:
\begin{multline*}
	\mathcal{L}^Nf(\boldsymbol{x}) = \sum_{i=1}^N \left[F(x_i,\mu_{\boldsymbol{x}}^N)\cdot \partial_if(\boldsymbol{x})+\frac{ 1}{2}\sum_{j,k=1}^da(x_i,\mu_{\boldsymbol{x}}^N)_{jk}\cdot \partial^2_{i}f(\boldsymbol{x})_{jk}\right.\\
	\left.+\lambda(x_i,\mu_{\boldsymbol{x}}^N)\int_{[0,1]^N}\left(f\left(\boldsymbol{x}+\Delta^N_i(x,\mu_{\boldsymbol{x}}^N,h^N)\right)-f(\boldsymbol{x})\right)\nu_N(dh^N)\right]
\end{multline*}
where $\partial_i f(\boldsymbol{x})$ indicates the vector of first order derivatives w.r.t.\ $x_i$, $\partial^2_i f(\boldsymbol{x})$ indicates the Hessian matrix of the second order derivatives w.r.t.\ $x_i$, $a(x_i,\mu_{\boldsymbol{x}}^N) \doteq \sigma(x_i,\mu_{\boldsymbol{x}}^N)\sigma(x_i,\mu_{\boldsymbol{x}}^N)^*$ and
\begin{equation*}
	\Delta^N_i(x,\mu_{\boldsymbol{x}}^N,h^N)_j \doteq \begin{cases}
	\frac{\Theta(x_i,x_j,\mu_{\boldsymbol{x}}^N,h^N_i,h^N_j)}{N}& \text{for } j\neq i,\\
\psi(x_i,\mu^N_{\boldsymbol{x}},h^N_i)& \text{for } j=i.
	\end{cases}
\end{equation*}

Towards a rigorous construction, allowing the limit as $N \to +\infty$, let us consider a filtered probability space $(\Omega,\mathcal{F},(\mathcal{F}_t)_{t\geq0},\Prb)$ satisfying the usual hypotheses, rich enough to carry an independent family $(B_i,\mathcal{N}^i)_{i\in\mathbb{N}}$ of $d$-dimensional Brownian motions $B_i$ and Poisson random measures $\mathcal{N}^i$ with characteristic measure $l\times l\times\nu$. Here $l$ is the Lebesgue measure restricted to $[0,\infty)$ and $\nu$ is a symmetric probability measure on $[0,1]^{\mathbb{N}}$  such that, for every $N \geq 1$, $\nu_N$ coincides with the projection of $\nu$ on the first $N$ coordinates. We will construct  $X^N$ as the solution of the following \mbox{SDE} 
\begin{multline}\label{SDE_XN}
	dX_i^N(t) = F(X_i^N(t),\mu_X^N(t))dt + \sigma(X^N_i(t),\mu_X^N(t))dB^i_t\\
	+ \frac{1}{N}\sum_{j\neq i} \int_{[0,\infty)\times [0,1]^{\mathbb{N}}} \Theta(X^N_j(t^-),X^N_i(t^-),\mu_X^N(t^-),h_j,h_i) \mathds{1}_{(0,\lambda(X^N_j(t^-),\mu^N_X(t^-))]}(u) \mathcal{N}^j(dt,du,dh)\\
	+\int_{[0,\infty)\times [0,1]^{\mathbb{N}}} \psi(X^N_i(t^-),\mu_X^N(t^-),h_i) \mathds{1}_{(0,\lambda(X^N_i(t^-),\mu_X^N(t^-))]}(u) \mathcal{N}^i(dt,du,dh),
\end{multline}
$i=1,\dots,N$. The existence and uniqueness of a solution starting from a vector of initial conditions $\left(X_1^N(0),\dots,X^N_N(0)\right)$ depends obviously on the assumptions on the coefficients, and we will specify sufficient conditions in the following sections.
\begin{remark} Notice that we made the choice of considering separately the jump's rate and amplitude. This is motivated by the fact that we are mainly interested in the jumps and we want to state a clear framework, useful for applications. The non-compensated jump component is often represented by a measure that does not directly describe the behavior of the system. Here, we want to highlight the role of the jumps, therefore we describe a diffusion process that at each position has a certain jump rate and a set of possible jumps, represented by the functions $\lambda$ and $\Delta^N$, respectively. The aim of this work is to give results without uniform boundedness assumptions on the jump rate. In the next sections, we will see that the first natural assumption is to have globally Lipschitz conditions on the functions $\lambda$ and $\Delta^N$. This is the reason why we need to perform all our proofs in an $L^1$ framework, instead of the classical $L^2$ approach for stochastic calculus. Indeed, when dealing with the well-posedness of the nonlinear Markov process, we will need to bound expectations of the supremum over a time interval of an integral w.r.t. the Poisson random measure
$\mathcal{N}$. In an $L^2$ framework, this involves the corresponding compensated martingale  $\tilde {\mathcal{N}}$ and needs bounds of the type, for $X,\,Y$ $\in$ $\mathbb{R}^d$,
\begin{equation}\label{cond}
\int_{0}^{\infty}\int_{[0,1]^{\mathbb{N}}}\|\Delta^N(X,h)\mathds{1}_{(0,\lambda(X)]}(u)-\Delta^N(Y,h)\mathds{1}_{(0,\lambda(Y)]}(u)\|^pdu\nu(dh)\leq C\|X-Y\|^p,
\end{equation}
for $p=2$. However, sometimes \eqref{cond} may hold for $p=1$, but not for $p=2$, which justifies the choice of getting the $L^1$ framework, where we do not need to compensate the process $\mathcal{N}$. For instance, if $\Delta$ is constant and $\lambda$ is globally Lipschitz, \eqref{cond} holds for $p=1$ and not $p=2$.
\end{remark}

\subsection{Macroscopic process}\label{macro}

Suppose the solution $X^N$ of \eqref{SDE_XN} exists, and that its initial condition has a permutation invariant distribution. Fix an arbitrary component $i$, and assume the process $X^N_i$ has a limit in distribution; by symmetry, the law of the limit does not depend on $i$, so we denote by $X$ the limit process. To identify, at a heuristic level, its law, we make the further assumption that a law of large numbers holds, i.e. $\mu^N_X(t)$ converges, as $N \to +\infty$, to the law $\mu_t$ of $X(t)$. Letting $N \to +\infty$ in \eqref{SDE_XN} we deduce, at a purely formal level, that the limit process $X(t)$ has the law of the solution of 
the {\em McKean-Vlasov} \mbox{SDE}:
\begin{multline}\label{SDE_limite}
	dX(t)=\left(F(X(t),\mu_t)+ \left\langle \mu_t,\lambda(\cdot,\mu_t)\int_{[0,1]^{2}}\Theta(\cdot,X(t^-),\mu_t,h_1,h_2)\nu_2(dh_1,dh_2) \right\rangle \right)dt\\
	+\sigma(X(t),\mu_t)dB_t +\int_{[0,\infty)\times[0,1]^{\mathbb{N}}} \psi(X(t^-),\mu_s,h_{1})\mathds{1}_{(0,\lambda(X(t^-),\mu_s)]}(u)\mathcal{N}(dt,du,dh).
\end{multline}
Here, $B$ is a $d_1$-dimensional Brownian motion and $\mathcal{N}$ an independent Poisson random measure with characteristic measure $dtdu\nu(dh)$ on $[0,\infty)^2\times[0,1]^{\mathbb{N}}$ as above. By $\langle \cdot,\cdot \rangle$ we indicate the integral of a function on its domain with respect to a certain measure; thus, $\langle \mu,\phi\rangle = \int_{\mathbb{R}^d} \phi(y)\mu(dy)$.

\begin{remark}
The Poisson random measures appearing in Equations \eqref{SDE_XN} and \eqref{SDE_limite}, respectively, have characteristic measure defined on $[0,\infty)^2\times [0,1]^{\mathbb{N}}$. The two equations could equivalently be stated in terms of Poisson random measures with characteristic measures defined on $[0,\infty)^2\times [0,1]^{N}$ (namely, $l\times l\times \nu_{N}$) and on $[0,\infty)^2\times [0,1] $ (namely, $l\times l\times \nu_1$). The reason for our seemingly unnatural choice is that it prepares for the coupling argument we will use below to establish propagation of chaos. We will need, for each $N$, a coupling of the $N$-particle system with $N$ independent copies of the limit system.
\end{remark}

Existence and uniqueness of solutions to \eqref{SDE_limite} starting from a given initial condition $X(0)$ will be discussed in the following sections. Note that \eqref{SDE_limite} is not a standard \mbox{SDE} since the law $\mu_t$ of the solution appears as an argument of its coefficients. It is often referred to as McKean-Vlasov \mbox{SDE}, as it is customary to call McKean-Vlasov equation the partial differential equation solved by the law $\mu_t$, namely, in the weak form,
\begin{equation*}
	\langle \mu_t,\phi\rangle - \langle \mu_0,\phi\rangle = \int_0^t \langle \mu_s,\mathcal{L}(\mu_s)\phi\rangle ds,
\end{equation*}
where
\begin{multline*}
	\mathcal{L}(\mu_t)\phi(x) \doteq F(x,\mu_t)\partial\phi(x) + \frac{1}{2}\sum_{j,k=1}^da(x,\mu_t)_{jk}\partial^2\phi(x)_{jk}\\
	+\left\langle \mu_t,\lambda(\cdot,\mu_t)\int_{[0,1]^{2}}\Theta(\cdot,x,\mu_t,h_1,h_2)\nu_2(dh_1,dh_2) \right\rangle \partial\phi(x)\\
	+\lambda(x,\mu_t)\int_{[0,1]}\left( \phi(x+\psi(x,\mu_t,h_1))-\phi(x)\right) \nu_1(dh_1).
\end{multline*}

\subsection{Propagation of chaos}\label{prop_chaos}

The link between the microscopic dynamics \eqref{SDE_XN} and the macroscopic limit \eqref{SDE_limite} is explained by the phenomenon of propagation of chaos. Let $\mu$ be a probability measure on $\R^d$. We assume that the sequence of the distributions of $X^N(0)$ is $\mu$-{\em chaotic}: for every $k \in \mathbb{N}$, the vector $(X^N_1(0), X^N_2(0), \ldots,  X^N_k(0))$ converges in distribution to the product measure $\mu^{\otimes k}$.  Fix an arbitrary time horizon $T>0$, and denote by $X^N[0,T] = (X^N(t))_{t\in[0,T]}$ the random path of the microscopic process up to time $T$. We say that {\em propagation of chaos} holds if the distribution of $X^N[0,T]$ is itself $Q$-chaotic for some probability measure $Q$ on the Skorohod space of {\em c{\`a}dl{\`a}g} functions $\mathbb{D}([0,T],\mathbb{R}^{d})$, that is, for each fixed $k\in \mathbb{N}$, the vector of random paths $(X^N_1[0,T], X^N_2[0,T], \ldots,  X^N_k[0,T])$ converges in distribution to $Q^{\otimes k}$. For a comprehensive introduction to the notion of propagation of chaos see \cite{Szn91}.

To illustrate the general strategy of proof, it is useful
to introduce an \emph{intermediate process} $Y^N=(Y^N(t))_{t\in[0,T]}$ with values in $\mathbb{R}^{d\times N}$. This Markov process $Y^N$ can be given as the solution of the \mbox{SDE}
\begin{multline}\label{SDE_YN}
	dY^N_i(t) = F(Y^N_i(t),\mu^N_Y(t))dt + \sigma(Y^N_i(t),\mu^N_i(t))dB^i_t \\
	+\frac{1}{N} \sum_{j=1}^{N} \lambda(Y^N_i(t^-),\mu^N_Y(t^-)) \int_{[0,1]^{2}} \Theta(Y^N_j(t^-),Y^N_i(t^-),\mu^N_Y(t^-),h_1,h_2)\nu_2(dh_1,dh_2) dt\\
	+ \int_{[0,\infty)\times[0,1]^{\mathbb{N}}} \psi(Y^N_i(t^-),\mu^N_Y(t^-),h)\mathds{1}_{(0,\lambda(Y^N_i(t^-),\mu^N_i(t^-))]}(u)\mathcal{N}^i(dt,du,dh),
\end{multline}
$i=1,\dots,N$, where again $B^i$ are independent d-dimensional Brownian motions and $\mathcal{N}^i$ are independent Poisson random measures with characteristic measure $l\times l\times\nu$. It is immediate to see that the process $Y^N$ differs from the original process $X^N$ in the jump terms; indeed, here the \emph{collateral jumps} have been absorbed by a new drift term, while the amplitude of the remaining jumps affects only one component a time. By using the {\em same} Brownian motions and the {\em same} Poisson random measures in \eqref{SDE_XN} and in \eqref{SDE_YN}, the processes $X^N$ and $Y^N$ are {\em coupled}, i.e. are realized on the same probability space: it will not be hard to give conditions for the $L^1$-convergence to zero of $X^N_1[0,T] - Y^N_1[0,T]$. Thus, the fact that the law of $X^N$ is $Q$-chaotic will follow if one shows that the law of $Y^N$ is $Q$-chaotic. Since $Y^N$ has no simultaneous jumps, this can be obtained along the lines of the classical approach. The intermediate process has the nice feature of highlighting the role of simultaneous jumps in the rate of convergence in $W_1$ Wasserstein distance of the empirical measure. Indeed by comparing the empirical measures of $X^N$ and $Y^N$, we obtain that the rate of convergence due to the simultaneous jumps is of the order $\frac{1}{\sqrt{N}}$, while the final rate obviously depends on the moments of initial conditions and of the process itself, see \cite{FoGu15}.


\section{Globally Lipschitz conditions on all coefficients}\label{Lipschitz_cond}

In this Section, we give Lipschitz conditions under which we can prove rigorously the results informally stated in the previous section. To state these conditions and the corresponding theorems, we need a suitable metric on spaces of probability measures.

Let $\mathcal{M}^1(\mathbb{R}^d)$ be the space of probability on $\R^d$ with finite first moment:
\[
	\mathcal{M}^1(\mathbb{R}^d) = \{ \mu \in \mathcal{M}(\mathbb{R}^d): \int \|x\| \mu(dx) < +\infty\}.
\]
This space is equipped with the $W_1$ Wasserstein metric:
\begin{align*}
	\rho(\mu,\nu)&\doteq \inf\left\{\int_{\mathbb{R}^d\times \mathbb{R}^d}\|x-y\|\pi(dx,dy);\; \pi \text{ has marginals } \mu \text{ and } \nu \right\}\\
	&=\sup\left\{\langle g,\mu\rangle - \langle g,\nu\rangle : g\!:\R^{d}\to \R,\; \|g(x)-g(y)\|\leq \|x-y\| \right\}.
\end{align*} 
We also consider the following subset of   $\mathcal{M}\left(\mathbf{D}\left([0,T],\mathbb{R}^d\right)\right)$, the set of the  probability measures on $\mathbf{D}\left([0,T],\mathbb{R}^d\right)$:
\[
	\mathcal{M}^1\left(\mathbf{D}\left([0,T],\mathbb{R}^d\right)\right) \doteq \left \{ \alpha \in \mathcal{M}\left(\mathbf{D}\left([0,T],\mathbb{R}^d\right)\right): \int_{\mathbf{D}}\sup_{t\in[0,T]}\|x(t)\|\alpha(dx) <+\infty \right\},
\]
and provide it with the metric
\begin{equation*}
	\rho_T(\alpha,\beta)\doteq \inf\left\{\int_{\mathbf{D}\times\mathbf{D}}\sup_{t\in[0,T]}\|x(t)-y(t)\|P(dx,dy);\, \text{ where } P \text{ has marginals }\alpha\text{ and }\beta\right\}.
\end{equation*}   

In what follows, we shall adopt a notion of chaoticity  which is stronger than the one we state above. 
\begin{definition} \label{def:ch}
	Let $X^{N} = (X^{N}_1, X^{N}_2, \ldots, X^{N}_N)$ be a sequence of random vectors with components $X^{N}_i \in \R^d$ (resp.\ $X^{N}_i \in \mathbf{D}\left([0,T],\mathbb{R}^d\right)$). For $\mu \in \mathcal{M}^1(\mathbb{R}^d)$ (resp.\ $\mu \in \mathcal{M}^1\left(\mathbf{D}\left([0,T],\mathbb{R}^d\right) \right)$), we say that $X^{N}$ is {\em $\mu$-chaotic in $W_1$} if its distribution is permutation invariant and, for each $k \in \mathbb{N}$, the law of the vector $ (X^{N}_1, X^{N}_2, \ldots, X^{N}_k)$ converges to $\mu^{\otimes k}$ with respect to the metric $\rho$ (resp.\ $\rho_T$).
\end{definition}

\subsection{Existence and uniqueness of solutions for the particle system and the McKean-Vlasov equation}

The conditions on the coefficients of system \eqref{SDE_XN} and the corresponding limit \eqref{SDE_limite} are as follows:
\begin{assumpt}\label{ASS_GLOB}
\begin{itemize}
	
\item[(Li)] The classical global Lipschitz assumption on $F$ and $\sigma$: $\exists\, \tilde L>0$ such that, for all $x,y \in \mathbb{R}^d$, all $\alpha,\gamma \in \mathcal{M}^{1}(\mathbb{R}^d)$,  
\begin{equation*}
	\|F(x,\alpha)-F(y,\gamma)\|\vee  \|\sigma(x,\alpha)-\sigma(y,\gamma)\|\leq \tilde L\left( \|x-y\|+\rho(\alpha,\gamma) \right). 
\end{equation*}

\item[(I)] 
The integrability condition: for all $N\in \mathbb{N}$, for all  $\mathbf{x} \in \mathbb{R}^{d\times N}$
\begin{equation*}
	\sup_{i\in \{1,\dots,N\}}\sup_{\alpha \in \mathcal{M}^1(\mathbb{R}^d)} \lambda(x_i,\alpha) \int_0^T \int_{[0,1]^{N}} \left\| \Delta^N_i(\mathbf{x},\alpha,h^N)\right\| \nu_N(dh^N)dt<\infty.
\end{equation*}

\item[(L1)] The $L^1$-Lipschitz assumption on the jump coefficients: $\exists$ $\bar{L} > 0$ such that, for all $x, y\in \mathbb{R}^{d}$, all $\alpha,\gamma$ $\in$ $\mathcal{M}^{1}(\mathbb{R}^d)$,
{\small
\begin{multline*}  
	\int_{[0,\infty)\times[0,1]} \|\psi(x,\alpha,h)\mathds{1}_{(0,\lambda(x,\alpha)]}(u) - \psi(y,\gamma,h)\mathds{1}_{(0,\lambda(y,\gamma)]}(u)\| du\nu_1(dh) 
	\leq \bar L \left(\|x-y\|+\rho(\alpha,\gamma)\right)
\end{multline*}}
and
{\small
\begin{multline*} 
	\| \langle \alpha,\lambda(\cdot,\alpha)\int_{[0,1]^{2}}\Theta(\cdot,x,\alpha,h_1,h_2)\nu_2(dh_1,dh_2) \rangle - \langle \gamma,\lambda(\cdot,\gamma)\int_{[0,1]^{2}}\Theta(\cdot,y,\gamma,h_1,h_2)\nu_2(dh_1,dh_2) \rangle \|\\
	\leq \bar{L} \left(\|x-y\|+\rho(\alpha,\gamma)\right),
\end{multline*}}

\end{itemize}
\end{assumpt}
In the following, we set $L\doteq \tilde L\vee \bar L$.

Existence and uniqueness of a square integrable strong solution of \eqref{SDE_XN} starting from a vector of square integrable initial conditions $\left(X_1^N(0),\dots,X^N_N(0)\right)$, independent of the family $(B_i,\mathcal{N}^i)_{i\in\mathbb{N}}$, are ensured by Assumption~\ref{ASS_GLOB}; see Theorem 1.2 in \cite{Gra92}. The same assumptions also guarantee existence and uniqueness of a strong solution of \eqref{SDE_limite} starting from any square integrable initial condition $X(0)$; see Theorem 2.1 in \cite{Gra92}.

\subsection{Propagation of chaos}

In addition to the aforementioned assumptions, for the proof of propagation of chaos, we will need the following square integrability condition on the amplitude of the collateral jumps:
\begin{assumpt}
\label{ASS_COLL_JUMPS}
\begin{itemize}
\item[(I2)]$\displaystyle{	\int_0^T \int_{[0,\infty)\times [0,1]^{N}} \|\Theta(x,y,\alpha,h_1,h_2) \mathds{1}_{(0,\lambda(x,\alpha)]}(u)\|^2 du\nu_2(dh)dt < \infty,}
$\\
for all $x,y \in \mathbb{R}^{d}$ and all $\alpha \in \mathcal{M}^1(\mathbb{R}^d)$.
\end{itemize}
\end{assumpt}

We begin by establishing the closeness between the original particle system $X^N$ and the intermediate process $Y^N$.

\begin{proposition}\label{iniziale_intermedio}
Grant Assumptions~\ref{ASS_GLOB} and \ref{ASS_COLL_JUMPS}. Let $X^N$ and $Y^N$ be the solutions of \eqref{SDE_XN} and \eqref{SDE_YN}, respectively. We assume the two processes are driven by the same Brownian motions and Poisson random measures, and start from the same square
integrable and permutation invariant initial condition. Then there exists a constant $C_T>0$ such that, for each fixed $i \in \mathbb{N}$, for all  $N\geq1$\begin{equation} \label{dist-interm}
	\Mean\left[ \sup_{t\in[0,T]} \|X^N_i(t)-Y^N_i(t)\| \right] \leq 
	\frac{C_T}{\sqrt{N}}.
\end{equation}
\end{proposition}

\begin{proof}
To simplify notation, we adopt the following abbreviations:
\begin{align*}
	\Theta_{i,j}(X^N(s^-),h)&\doteq \Theta(X^N_i(s^-),X^N_j(s^-),\mu^N_X(s^-),h_i,h_j), \\
	\lambda_{i}(X^N(s^-))&\doteq \lambda(X^N_i(s^-),\mu^N_X(s^-)),\\
	\psi_i(X^N(s^-),h)&\doteq \psi(X^N_i(s^-),\mu^N_X(s^-),h_i),\\
	U&\doteq [0,\infty)\times [0,1]^{\mathbb{N}}.
\end{align*}
By permutation invariance of both the initial condition and the dynamics, we have, for every $t\in [0,T]$,
\[
	\Mean\left[\sup_{s\in[0,t]} \|X^N_i(s)-Y^N_i(s)\|\right] = \frac{1}{N}\sum_{j=1}^{N} \Mean\left[\sup_{s\in[0,t]}\|X^N_j(s)-Y^N_j(s)\|\right].
\]

Fix $t\in [0,T]$, and set
\begin{align*}
	F_i&\doteq \Mean\left[\int_0^t \|F(X^N_i(s),\mu^N_X(s))-F(Y^N_i(s), \mu^N_Y(s))\|ds\right],\\
	\sigma_i&\doteq \Mean\left[\sup_{r\in[0,t]}\left\| \int_0^r \left( \sigma(X^N_i(s),\mu^N_X(s))-\sigma(Y^N_i(s),\bar\mu^N_Y(s)) \right) dB^i_s \right\|\right],\\
\begin{split}
	\Theta_i&\doteq \Mean\left[ \sup_{r\in[0,t]}\left\|\frac{1}{N}\sum_{j\neq i} \int_{[0,r]\times U} \Theta_{j,i}(X^N(s^-),h)\mathds{1}_{(0,\lambda_j(X^N(s^-))]}(u) \mathcal{N}^j(ds,du,dh)\right.\right. \\
	&\qquad \left.\left. - \frac{1}{N}\sum_{j=0}^N \int_{[0,t]\times U}\Theta_{j,i}(Y^N(s^-),h)\mathds{1}_{(0,\lambda_j(Y^N(s^-))]}(u)ds \,du\nu(dh)\right\|\right],
\end{split}\\
\begin{split}
	\psi_i&\doteq \Mean \left[ \sup_{r\in[0,t]} \left\| \int_{[0,r]\times U}  \psi_i(X^N(s^-),h)\mathds{1}_{(0,\lambda_i(X^N(s^-))]}(u) \mathcal{N}^i(ds,du,dh) \right.\right. \\
	&\qquad \left.\left. - \int_{[0,r]\times U} \psi_i(Y^N(s^-),h)\mathds{1}_{(0,\lambda_i(Y^N(s^-))]}(u)  \mathcal{N}^i(ds,du,dh) \right\|\right].
\end{split}
\end{align*}
Note that all these quantities do not depend on $i$, that is therefore omitted in what follows.
Then
\begin{equation}\label{distanza_sup}
	\Mean\left[\sup_{s\in[0,t]}\|X^N_i(s)-Y^N_i(s)\|\right] \leq F+\sigma+\Theta+\psi.
\end{equation}

The term $F$ can be easily bounded thanks to the Lipschitz condition (Li) and the coupling bound for the $W_1$ Wasserstein metric, and we obtain
\begin{equation*}
	F\leq L\int_0^{t} \Mean\left[\|X^N_i(s)-Y^N_i(s)\|\right] ds+\frac{L}{N}\sum_{j=1}^N\int_0^{t}\Mean\left[\|X^N_j(s)-Y^N_j(s)\|\right] ds.
\end{equation*}
The bound on $\sigma$ involves the Burkholder-Davis-Gundy inequality, and we get, for some constant $M$ not depending on $N$ nor $t$,
\begin{align*}
	\sigma&\leq M \Mean\left[ \left(\int_0^{t} \left(\|X^N_i(s)-Y^N_i(s)\|+\frac{1}{N}\sum_{j=1}^N\|X^N_j(s)-Y^N_j(s)\|\right)^2ds \right)^{1/2} \right] \\
	&\leq M \sqrt{t} \Mean\left [\sup_{s\in[0,t]} \|X^N_i(s)-Y^N_i(s)\| + \frac{1}{N} \sum_{j=1}^N \sup_{s\in[0,t]}\|X^N_j(s)-Y^N_j(s)\|\right].
\end{align*}
The term $\Theta$ needs to be treated again with the Burkholder-Davis-Gundy inequality. In what follows, we denote by $\tilde{\mathcal{N}}^i$ the compensated Poisson measure associated to $\mathcal{N}^i$ and it is crucial the fact that $\{\tilde{\mathcal{N}}^i\}_{i=1,\dots,N}$ is a family of orthogonal martingales. Therefore, for a certain constant $K>0$ coming from the Burkholder-Davis-Gundy inequality, the constant $L>0$ coming from condition (L1) and a constant $C>0$ not depending on $N$ nor $t$, we have{\footnotesize
\begin{align*}
	\Theta&\leq \Mean\left[ \sup_{r\in[0,t]}\left\|\frac{1}{N}\sum_{j\neq i}\int_{[0,r]\times U} \Theta_{j,i}(X^N(s^-),h)\mathds{1}_{[0,\lambda_j(X^N(s^-)))} \tilde{\mathcal{N}}^j(ds,du,dh)\right\|\right]\\
	&+ \Mean\left[ \sup_{r\in[0,t]}\left\|\frac{1}{N}\sum_{j=1}^N \int_{[0,r]\times U} \left( \Theta_{j,i}(X^N(s^-),h)\mathds{1}_{(0,\lambda_j(X^N(s^-))]} - \Theta_{j,i}(Y^N(s^-),h)\mathds{1}_{(0,\lambda_j(Y^N(s^-))]} \right)ds\, du\nu(dh)\right\|\right]\\
	&+ \frac{1}{N} \Mean\left[ \sup_{r\in[0,t]} \left\|\int_{[0,r]\times U} \Theta_{i,i}(X^N(s^-),h)\mathds{1}_{(0,\lambda_j(X^N(s^-))]}ds\,du\,\nu(dh)\right\|\right]\\
	&\leq \frac{K}{N} \Mean\left[ \left( \sum_{j\neq i} \int_0^{t}\int_U \left\|\Theta_{j,i}(X^N(s^-),h)\mathds{1}_{(0,\lambda_j(X^N(s^-))]}(u)\right\|^2 ds\,du\nu(dh) \right)^{1/2} \right] \\
\begin{split}
	&+ \int_0^{t} \Mean\left[\left\| \left\langle \mu^N_s, \int_U\Theta_{\cdot,i}(X^N(s^-),h)\mathds{1}_{[0,\lambda_{\cdot}(X^N(s^-)))}(u)\, du\nu(dh) \right\rangle \right.\right.\\
	&\qquad \left.\left. - \left\langle \bar{\mu}^N_s, \int_U\Theta_{\cdot,i}(Y^N(s^-),h)\mathds{1}_{[0,\lambda_{\cdot}(Y^N(s^-)))}(u)\, du\nu(dh)\right\rangle \right\|\right]ds
\end{split}\\
	&+ \frac{1}{N} \Mean\left[ \int_0^{t} \int_U \left\|\Theta_{i,i}(X^N(s^-),h)\mathds{1}_{[0,\lambda_i(X^N(s^-)))}(u)\right\| du\nu(dh) ds\right]\\
	&\leq \frac{C}{\sqrt{N}}+L\int_0^{t} \Mean\left[\|X^N_i(s)-Y^N_i(s)\|\right]ds +\frac{L}{N}\sum_{j=1}^N\int_0^{t} \Mean\left[\|X^N_j(s)-Y^N_j(s)\|\right] ds+\frac{C}{N}.
\end{align*}}
The term $\psi$ concerns the main jumps of the particle system and is bounded by the positivity property of Poisson processes and the Lipschitz condition (L1):{\small
\begin{align*}
	\psi &\leq \Mean\left[\int_{[0,t]\times U} \left\|\psi_i(X^N(s^-),h)\mathds{1}_{(0,\lambda_i(X^N(s^-))]}(u)- \psi_i(Y^N(s^-),h)\mathds{1}_{(0,\lambda_i(Y^N(s^-))]}(u) \right\| \mathcal{N}^i(ds,du,dh)\right]\\
	&= \Mean\left[ \int_{[0,t]\times U} \left\| \psi_i(X^N(s^-),h)\mathds{1}_{(0,\lambda_i(X^N(s^-))]}(u) - \psi_i(Y^N(s^-),h)\mathds{1}_{(0,\lambda_i(Y^N(s^-))]}(u) \right\|ds\, du\nu(dh)\right]\\
	&\leq L\int_0^{t} \Mean\left[\|X^N_i(s)-Y^N_i(s)\|\right]ds + \frac{L}{N}\sum_{j=1}^N \int_0^{t}\Mean\left[\|X^N_j(s)-Y^N_j(s)\|\right]ds.
\end{align*}}

Therefore, recalling \eqref{distanza_sup}, we find that, for every $t\in [0,T]$, 
\begin{multline*}
	\Mean \left[ \sup_{s\in[0,t]}\|X^N_i(s)-Y^N_i(s)\|\right] \\
	\leq M\sqrt{t} \Mean\left[ \sup_{s\in[0,t]}\|X^N_i(s)-Y^N_i(s)\| \right] + M\sqrt{t} \Mean\left[ \frac{1}{N}\sum_{j=1}^N \sup_{s\in[0,t]} \|X^N_j(s)-Y^N_j(s)\| \right] \\
	+ 3L\int_0^{t} \Mean\left[\|X^N_i(s)-Y^N_i(s)\|\right]ds + \frac{3L}{N} \sum_{j=1}^N \int_0^{t} \Mean\left[\|X^N_j(s)-Y^N_j(s)\|\right] ds + \frac{C}{N}+\frac{C}{\sqrt{N}}.
\end{multline*}
Choose $T_0 > 0$ small enough so that $(1-2M\sqrt{T_0})>0$. By summing over the index $i$ in the above inequality and dividing both sides by $N$, we can move the first two terms on the right-hand side to the left, obtaining, for every $t\in [0,T_{0}]$,
\begin{equation*}
\begin{split}
	\frac{1}{N} \sum_{i=1}^N \Mean\left[ \sup_{s\in[0,t]} \|X^N_i(s)-Y^N_i(s)\|\right] &\leq \frac{6K}{1-2M\sqrt{t}} \int_0^t \frac{1}{N}\sum_{i=1}^N\Mean\left[\sup_{s\in[0,r]}\|X^N_i(s)-Y^N_i(s)\|\right] dr\\
	&\quad + \frac{C}{N(1-2M\sqrt{t})} + \frac{C}{\sqrt{N}(1-2M\sqrt{t})}.
\end{split}
\end{equation*}
An application of Gronwall's lemma yields 
\begin{equation}\label{distanza_zero}
	\frac{1}{N} \sum_{i=1}^N \Mean\left[\sup_{t\in[0,T_0]}\|X^N_i(t)-Y^N_i(t)\|\right] \leq \frac{C_{T_0}}{\sqrt{N}}
\end{equation}
for some finite constant $C_{T_{0}}$ not depending on $N$. Recall that \eqref{distanza_zero} holds on a time interval $[0,T_0]$ for $T_0$ sufficiently small. If $T_0$ is smaller than $T$, then we can repeat the procedure of estimates on the interval $[T_0,(2T_0)\wedge T]$. In this case, we find that, for every $t\in [T_{0},(2T_0)\wedge T]$, 
\begin{align*}
	\frac{1}{N}\sum_{i=1}^N \Mean\left[\sup_{s\in[T_0,t]}\|X^N_i(s)-Y^N_i(s)\|\right]
	&\leq \frac{1}{1-2M\sqrt{t-T_0}} \left(\frac{1}{N}\sum_{i=1}^N \Mean\left[\sup_{s\in[0,T_0]}\|X^N_i(s)-Y^N_i(s)\|\right]\right)\\
	&+ \frac{6K}{1-2M\sqrt{t-T_0}} \int_{T_0}^{t} \frac{1}{N}\sum_{i=1}^N \Mean\left[\sup_{s\in [T_0,r]} \|X^N_i(s)-Y^N_i(s)\|\right]dr \\
	&+ \frac{C}{N(1-2M\sqrt{t-T_0})} +\frac{C}{\sqrt{N}(1-2M\sqrt{t-T_0})},
\end{align*}
where the first term comes from a bound on the initial condition ${\displaystyle \frac{1}{N}\sum_{i=1}^N \Mean\left[\|X^N_i(T_0)-Y^N_i(T_0)\|\right]}$. Hence, again thanks to Gronwall's lemma, for some constant $C_{2,T_0}$,
\begin{equation*}	
	\frac{1}{N}\sum_{i=1}^N \Mean\left[\sup_{s\in[0,(2T_0)\wedge T]} \|X^N_i(s)-Y^N_i(s)\|\right] \leq \frac{C_{2,T_0}}{\sqrt{N}}.
\end{equation*}
We proceed by induction until we cover, after finitely many steps, the entire interval $[0,T]$. By exchangeability of the laws of both the initial and the intermediate process, this yields, for $i=1,\dots,N$ 
$$
\Mean\left[\sup_{s\in [0,T]}\left\|X^N_i(s)-Y^N_i(s)\right\|\right]\leq \frac{C_T}{\sqrt{N}}
$$
  and \eqref{dist-interm} holds.
\end{proof}

In the next, we use a similar coupling technique and we now show propagation of chaos for $Y^N$.

\begin{proposition}\label{intermedio_limite}
Grant Assumptions~\ref{ASS_GLOB} and \ref{ASS_COLL_JUMPS}. Let $\mu_0$ be a probability measure on $\R^d$ such that $\int \|x\|^2 \mu_0(dx)<+\infty$. For $N\in \mathbb{N}$, let $Y^N$ be a solution of Eq.~\eqref{SDE_YN} in $[0,T]$. Assume that $Y^N(0) = (Y^N_1(0), \ldots, Y^N_N(0))$, $N\in \mathbb{N}$, form a sequence of square integrable random vectors that is $\mu$-chaotic in $W_1$. Let $\mu$ be the law of the solution of Eq.~\eqref{SDE_limite} in $[0,T]$ with initial law $\Prb \circ X(0)^{-1} = \mu_0$. Then $Y^N$ is $\mu$-chaotic in $W_1$.
\end{proposition}

\begin{proof}
In order to get the thesis, we set a coupling procedure. Let the processes $Y^N_i$, $N\in \mathbb{N}$, $i\in \{1,\ldots,N\}$ be all defined on the filtered probability space $(\Omega,\mathcal{F},(\mathcal{F}_t)_{t\geq0},\Prb)$ with respect to the family $(B_i,\mathcal{N}^i)_{i\in\mathbb{N}}$ of Brownian motions and Poisson random measures. Since $(Y^N(0))$ is $\mu$-chaotic in $W_1$ by hypothesis, we assume, as we may, that our stochastic basis carries a triangular array $(\bar{X}^N_{i}(0))_{i\in \{1,\ldots,N\}, N\in \mathbb{N}}$ of identically distributed $\R^d$-valued random variables with common distribution $\mu$ such that $(\bar{X}^N_{i}(0))_{i\in \{1,\ldots,N\}, N\in \mathbb{N}}$ and $(B_i,\mathcal{N}^i)_{i\in\mathbb{N}}$ are independent, the sequence $(\bar{X}^N_{i}(0))_{i\in \{1,\ldots,N\}}$ is independent for each $N$, and
\[
	\phi^N \doteq \Mean\left[ \left\|\bar X^N_i(0) - Y^N_i(0)\right\|\right] 
\]
tends to zero as $N\to\infty$.

For $N\in \mathbb{N}$, $i\in \{1,\ldots,N\}$, let $\bar X^N_i$ be the unique strong solution of Eq.~\eqref{SDE_limite} in $[0,T]$ with initial condition $X^N_{i}(0)$, driving Brownian motion $B_i$ and Poisson random measure $\mathcal{N}^i$. Notice that the processes $X^N_1,\ldots, X^N_N$ are independent and identically distributed for each $N$.

By definition of the metric $\rho_T$, the $\mu$-chaoticity in $W_1$ of the sequence $Y^N$ follows from
\begin{equation}\label{limit}
	\lim_{N\rightarrow\infty} \Mean\left[ \sup_{t\in[0,T]}\left\|\bar X^N_i(t) - Y^N_i(t)\right\|\right] =0,
\end{equation}
 for every fixed $i\in\mathbb{N}$. However, the limit is the same by exchangeability of components. The term in \eqref{limit} is bounded by 
\begin{align*}
	&\Mean\left[\sup_{t\in[0,T]}\|Y^N_i(t)-\bar X^N_i(t)\|\right]\leq \phi^N + \bar F + \bar \sigma + \bar \Theta + \bar \psi, \\
\intertext{where}
	\bar F &\doteq \Mean\left[\int_0^T \|F(Y^N_i(s),\mu_Y^N(s))-F(\bar X^N_i(s),\mu_s)\|ds\right],\\
	\bar \sigma &\doteq \Mean\left[\sup_{t\in[0,T]}\left\|\int_0^t \sigma(Y^N_i(s),\mu_Y^N(s))-\sigma(\bar X^N_i(s),\mu_s)dB^i_s\right\|\right],\\
\begin{split}
	\bar \Theta &\doteq \Mean\left[\sup_{t\in[0,T]}\left\|
\int_0^t \left\langle \mu^N_Y(s), \int_{U}\Theta(\cdot, Y_i^N(s),\mu^N_Y(s),h) \mathds{1}_{(0,\lambda_j(\cdot,\mu^N_Y(s))]}(u)\, du\nu_2(dh)\right\rangle ds \right.\right.\\
&\qquad - \left.\left. \int_{0}^t\int_{U} \left\langle \mu_s, \Theta(\cdot,\bar X_i^N(s),\mu_{s},h) \mathds{1}_{(0,\lambda_j(\cdot,\mu_{s})]}(u)\, du\nu(dh)\right\rangle ds\right\|\right],
\end{split}\\
\begin{split}
	\bar \psi &\doteq \Mean\left[\sup_{t\in[0,T]}\left\|\int_{[0,t]\times U}\psi(Y_i^N(s^-),\mu^N_Y(s),h)\mathds{1}_{(0,\lambda(Y_i^N(s^-),\mu^N_Y(s^-))]}(u) \right.\right.\\
	&\qquad \left.\left. - \psi(\bar X^N_i(s^-),\mu_{s^-},h)\mathds{1}_{(0,\lambda_i(\bar X_i^N(s^-),\mu_{s^-})]}(u)\mathcal{N}^i(dt,du,dh)\right\|\right].
\end{split}
\end{align*}
The terms $\bar F$, $\bar \sigma$, and $\bar \psi$ are treated exactly as in Proposition~\ref{iniziale_intermedio}, whereas the term $\bar \Theta$ only requires the application of the Lipschitz condition (L1). By mimicking the steps in Proposition~\ref{iniziale_intermedio}, there exists a $T_0>0$ small enough and a constant $C_{T_0}\geq 0$, independent of $N$, such that we can apply Gronwall's Lemma and obtain 
\begin{equation*}
	\Mean\left[\sup_{t\in[0,T_0]}\|Y^N_i(t)-\bar X^N_i(t)\|\right] \leq C_{T_0} \left(\int_0^{T_0}\Mean\left[\rho(\mu^N_{Y}(t),\mu_t)\right]dt+\sqrt{\int_0^{T_0}\Mean[\rho(\mu^N_Y(t),\mu_t)^2]dt}+\phi^N\right).
\end{equation*}
By triangular inequality, for every fixed $t\,\in\, [0,T_0]$,
{\small
\[
\begin{split}
\Mean\left[\rho(\mu^N_{Y}(t),\mu_t)\right] & \leq \Mean\left[\rho(\mu^N_{Y}(t),\mu^N_{\bar{X}}(t))\right]+\Mean\left[\rho(\mu^N_{\bar{X}}(t),\mu_t)\right] \\
& \leq \Mean\left[\sup_{t\in[0,T_0]}\|Y^N_i(t)-\bar X^N_i(t)\|\right]+\Mean\left[\rho(\mu^N_{\bar{X}}(t),\mu_t)\right]
\end{split} \] }
Then, for a $T_0$ sufficiently small, using again Gronwall Lemma, there exists a positive constant, depending on $T_0$, that by abuse of notation we will indicate it again with $C_{T_0}>0$,  such that
\begin{multline*}
\Mean\left[\sup_{t\in[0,T_0]}\|Y^N_i(t)-\bar X^N_i(t)\|\right] \\ \leq C_{T_0}\left(\int_0^{T_0} \Mean\left[\rho(\mu^N_{\bar{X}}(t),\mu_t)\right]dt+\sqrt{\int_0^{T_0} \Mean\left[\rho(\mu^N_{\bar{X}}(t),\mu_t)^2\right]dt}+\phi^N\right).
\end{multline*}
We see that the bound on \eqref{limit} depends on the initial conditions and on $\Mean\left[\rho(\mu^N_{\bar{X}}(t),\mu_t)\right]$, that is the distance, at every fixed time $t\in[0,T]$,   between the empirical measure of $N$ i.i.d. copies of the solution of the process with law $\mu$ and the law $\mu_t$ itself. The rate of convergence of empirical measures in Wasserstein distance depends on the moments of $\bar{X}(t)$ and on the dimension $d$, see Theorem 1 in \cite{FoGu15}, for a complete characterization.   
Since
\[
\sup_{t \in [0,T]} \Mean\left[\bar{X}_i^2(t) \right] < +\infty,
\]
it follows from \cite{FoGu15} that, setting
\[
\beta^N := \sup_{t \in [0,T]}\Mean[\rho(\mu^N_{\bar{X}}(t),\mu_t)],
\]
we have
\[
\lim_{N\rightarrow\infty} \beta^N = 0.
\]
Therefore, we know that there exists a constant $C_{T_0}>0$ such that, for $N$ going to infinity, we have 
$$
\Mean\left[\sup_{t\in[0,T_0]}\|Y^N_i(t)-\bar X^N_i(t)\|\right]\leq C_{T_0}\left(\beta^N+\phi^N\right).
$$ 
Iterating this procedure as in Proposition~\ref{iniziale_intermedio}, we extend the above result to $[0,T]$, i.e.
$$
\Mean\left[\sup_{t\in[0,T]}\|Y^N_i(t)-\bar X^N_i(t)\|\right]\leq C_{T}\left(\beta^N+\phi^N\right).
$$ 
for a suitable constant $C_T$.
 This establishes $\mu$-chaoticity of $Y^N$ in $W_1$.
 
\end{proof}

\begin{remark} \label{rateconv}
By the results in \cite{FoGu15}, $\beta^N = O\left(\frac{1}{\sqrt{N}}\right)$ except possibly for dimensions $d=1,2$, where logarithmic corrections may appear. In Proposition~\ref{iniziale_intermedio} we prove that, in any situation, the simultaneous jumps in  the form presented here, do not worsen this rate of convergence, since they add a term of order $\frac{1}{\sqrt{N}}$. Note that if the components of the initial condition are i.i.d., then also $\phi^N = O\left(\frac{1}{\sqrt{N}}\right)$, and is this case we get, for some $C_T>0$,
\[
\Mean\left[\sup_{t\in[0,T]}\|Y^N_i(t)-\bar X^N_i(t)\|\right]\leq \frac{C_{T}}{\sqrt{N}}.
\]

\end{remark}

The propagation of chaos property for $X^N$ is now an immediate consequence of Propositions \ref{iniziale_intermedio} and \ref{intermedio_limite}.

\begin{corollary}\label{propch_limite}
	Grant Assumptions~\ref{ASS_GLOB} and \ref{ASS_COLL_JUMPS}. Let $\mu_0$ be a probability measure on $\R^d$ such that $\int \|x\|^2 \mu_0(dx)<+\infty$. For $N\in \mathbb{N}$, let $X^N$ be a solution of Eq.~\eqref{SDE_XN} in $[0,T]$. Assume that $X^N(0) = (X^N_1(0),\ldots, X^N_N(0))$, $N\in \mathbb{N}$, form a sequence of square integrable random vectors that is $\mu_0$-chaotic in $W_1$. Let $\mu$ be the law of the solution of Eq.~\eqref{SDE_limite} in $[0,T]$ with initial law $\Prb \circ X(0)^{-1} = \mu_0$. Then $X^N$ is $\mu$ chaotic in $W_1$.
\end{corollary}


\section{Non-globally Lipschitz drift}\label{non_lip_drift}

We are interested in enlarging the class of systems for which propagation of chaos holds. In this Section we relax the Lipschitz assumption on the drift, allowing gradients of general convex potentials. This includes relevant examples as those appeared in \cite{Da83} and \cite{gar13}.

\subsection{Particle system} 
Consider the particle system \eqref{SDE_XN} in Section~\ref{Particle_system_and_limit}. The coefficients are supposed to satisfy the following set of conditions.  
\begin{assumpt}\label{ASS_NL2} \begin{itemize} \item[(U)] The drift coefficient $F\!:\mathbb{R}^d\times \mathcal{M}(\mathbb{R}^d)\rightarrow \mathbb{R}^d$ is of the form 
\begin{equation*}
	F(x,\alpha) = -\triangledown  U(x) + b(x,\alpha),
\end{equation*}
for all $x \in\mathbb{R}^d$ and all $\alpha \in \mathcal{M}(\mathbb{R}^d)$,
where $U$ is convex and $\mathcal{C}^1$. The function $b$ is assumed to be globally Lipschitz in both variables,
 and for all $x\,\in\,\mathbb{R}^d$  we have $\sup_{\alpha\in \mathcal{M}^1(\mathbb{R}^d)}b(x,\alpha)<\infty$.
\item[(LD)] The diffusion coefficient $\sigma\!:\mathbb{R}^d\times\mathcal{M}(\mathbb{R}^d)\rightarrow \mathbb{R}^{d\times d_1}$ satisfies the usual global Lipschitz condition, i.e.,
$\exists\,\tilde L > 0$ such that for all $x,y\in \mathbb{R}^d$, all $\alpha,\gamma\in \mathcal{M}(\mathbb{R}^d)$,
\begin{equation*}
	\|\sigma(x,\alpha)-\sigma(y,\gamma)\|	\leq \tilde L \left(\|x-y\|+\rho(\alpha,\gamma)\right).
\end{equation*}
Moreover, for all $x\,\in\,\mathbb{R}^d$ $\sup_{\alpha\in\mathcal{M}^1(\mathbb{R}^d)}\sigma(x,\alpha)<\infty$. 
\end{itemize}
\noindent The jumps' coefficients satisfies conditions $(I)$ and $(L1)$,  from Section \ref{macro}.
\end{assumpt}
\noindent As before, we set $L\doteq \bar L \vee \tilde L$.\\

\begin{remark}
Condition $(U)$ is a natural choice when one wants to relax globally-Lipschitz conditions on coefficients. It induces a process whose trajectories are strongly constrained by the convex potential. This attracting drift, even when combined with an unbounded jump rate, should prevent the process from exploding in finite time. We will see that this is what happens provided the jump rate is in some way ``controllable'', as it is under the Lipschitz assumption $(L1)$.
\end{remark}

\subsection{McKean-Vlasov equation with non-Lipschitz drift}
\label{non-standard-drift}

Consider the stochastic differential equation
\begin{multline} \label{Mckean-vlasov_non_lipschitz}
	dX(t) = F(X(t),\mu_{t})dt + \sigma(X(t),\mu_{t})dB_t\\
	+ \int_{[0,\infty)\times[0,1]^{\mathbb{N}}} \psi(X(t^-),\mu_{t^-},h_1) \mathds{1}_{(0,\lambda(X(t^-),\mu_{t^-})]}(u) \mathcal{N}(dt,du,dh),
\end{multline} 
where $\mu_{t} = \Law(X(t))$, $B$ is a $d_1$-dimensional Brownian motion and  $\mathcal{N}$ stationary Poisson random measures with characteristic measure $l\times l\times \nu$.

\begin{theorem}\label{existence_uniqueness_McKeanVlasov_NONSTANDARD_DRIFT}
Let the coefficients of the nonlinear \mbox{SDE}~\eqref{Mckean-vlasov_non_lipschitz} satisfy
Assumption~\ref{ASS_NL2}. Then for all square integrable initial conditions $X(0)$~$\in$~$\mathbb{R}^d$, Eq.~\eqref{Mckean-vlasov_non_lipschitz} admits a unique strong solution.
\end{theorem}
\begin{proof}

 Let $P^1$ and $P^2$ two laws on $\mathbf{D}([0,T],\mathbb{R}^d)$ and suppose that $X^1$ and $X^2$ are two solutions of the following \mbox{SDE}, for $k=1,2$:
\begin{multline} \label{Mckean-vlasov_non_lipschitz_uniqueness}
	dX^k(t) = F(X^k(t),P^k_{t})dt + \sigma(X^k(t),P^k_{t})dB_t\\
	+ \int_{[0,\infty)\times[0,1]^{\mathbb{N}}} \psi(X^k(t^-),P^k_{t^-},h_1) \mathds{1}_{(0,\lambda(X^k(t^-),P^k_{t^-})]}(u) \mathcal{N}(dt,du,dh),
\end{multline} 
 defined on the same probability space $\left(\Omega,\mathcal{F},(\mathcal{F}_t),\Prb\right)$  with the same $\mathcal{F}_t$-Brownian motion $B$, the same  Poisson random measure $\mathcal{N}$ and with initial condition $X^1(0) = X^2(0) = \xi$ $\Prb$-almost surely. The well-posedness of Eq.~\eqref{Mckean-vlasov_non_lipschitz_uniqueness} is ensured by Lemma~\ref{lemma_SDE_parametrizzata}. Let $Q^1$ and $Q^2$ be the laws of the solutions on $\mathbf{D}([0,T),\mathbb{R}^d)$ and let $\Gamma$ be the map that associates $Q^k$ to $P^k$.  We are interested in proving that the map $\Gamma$ is a contraction for the $W_1$ Wasserstein norm. Hence, we want to bound the distance 
\begin{equation}\label{differenza}
\rho_T(Q^1,Q^2)\leq\mathbf{E}\left[\sup_{t\in[0,T]}\|X^1(t)-X^2(t)\|\right].
\end{equation} 
The idea here, in order to exploit the convexity of $U$, is to apply Ito's rule. A classical approach consists in applying Ito's rule to a quantity of type $(X^1_t-X^2_t)^2$; this $L^2$ approach does not work in presence of jump terms. For this reason we rather use a $L^1$ approach. To this aim, for all $\epsilon>0$ 
we define the following smooth approximation of the norm $$f^{\epsilon}(x)\doteq \|x\|\mathds{1}(\|x\|>\epsilon)+\left(\frac{\|x\|^2}{2\epsilon}+\frac{\epsilon}{2}\right)\mathds{1}(\|x\|\leq \epsilon).$$
Then, by Ito's rule and Fatou's Lemma, we have 
\begin{multline*}
	\Mean\left[\sup_{t\in[t_0,t_1]}\|X^1(t)-X^2(t)\|\right] \leq \liminf_{\epsilon\downarrow 0}\Mean\left[\sup_{t\in[t_0,t_1]}f^{\epsilon}\left(X^1(t)-X^2(t)\right)\right] \\
	 \leq \liminf_{\epsilon\downarrow 0}\left(i_{\epsilon}[t_0,t_1] + u_{\epsilon}[t_0,t_1]  + b_{\epsilon}[t_0,t_1] +\sigma_{\epsilon}[t_0,t_1]+\Sigma_{\epsilon}[t_0,t_1]+\Lambda_{\epsilon}[t_0,t_1] \right),
\end{multline*}
where, for $t_1\in [t_0,T]$, we set
{\small \begin{align*}
	i_{\epsilon}[t_0,t_1]&\doteq \Mean\left[f^{\epsilon}\left(X^1(t_0)-X^2(t_0)\right) \right],\\
	u_{\epsilon}[t_0,t_1]&\doteq \Mean\left[\sup_{t\in[t_0,t_1]} -\int_{t_0}^t\triangledown f^{\epsilon}\left(X^1(s)-X^2(s)\right) \cdot \triangledown\left( U(X^1(s))-U(X^2(s))\right)ds  \right],\\
	b_{\epsilon}[t_0,t_1]&\doteq \Mean\left[\sup_{t\in[t_0,t_1]} \int_{t_0}^t\triangledown f^{\epsilon}\left(X^1(s)-X^2(s)\right) \cdot\left(b(X^1(s),P^1_s)-b(X^2(s),P^2_s)\right)ds  \right],\\
	\sigma_{\epsilon}[t_0,t_1]&\doteq\displaystyle{ \frac{1}{2}\Mean\Bigg[\sup_{t\in[t_0,t_1]} \int_{t_0}^t  }\\
	 & \displaystyle{ \qquad \qquad  Tr\Big[\left(\sigma(X^1(s),P_s^1)-\sigma(X^2(s),P_s^2) \right)^T H_{f^{\epsilon}(X^1(s)-X^2(s))}\left(\sigma(X^1(s),P_s^1)-\sigma(X^2(s),P_s^2) \right) \Big] ds \Bigg],}\\
	\Sigma_{\epsilon}[t_0,t_1]&\doteq  \Mean\left[\sup_{t\in[t_0,t_1]} \int_{t_0}^{t} \triangledown f^{\epsilon}\left(X^1(s)-X^2(s)\right)\cdot \left(\sigma(X^1(s),P_s^1)-\sigma(X^2(s),P_s^2)\right)dB_s \right],\\
	\Lambda_{\epsilon}[t_0,t_1]&\displaystyle{\doteq \Mean\Bigg[\sup_{t\in[t_0,t_1]} \int_{t_0}^t\int_{[0,1]}\int_0^{\infty} f^{\epsilon}\left(X^1(s)+\psi(X^1(s),P^1_s,h)\mathds{1}_{u\leq\lambda(X^1(s),P^1_s)}-X^2(s)\right.}\\
	&\displaystyle{ \qquad \qquad \qquad \qquad \qquad\qquad\left.-\psi(X^2(s),P^2_s,h)\mathds{1}_{u\leq\lambda(X^2(s),P^2_s)}\right)-f^{\epsilon}\left(X^1(s)-X^1(s)\right)dsdu \nu_1(dh)\Bigg].}
\end{align*}}
Notice that, by the assumption of convexity of $U$, for all $\mathbf{x}$ and $\mathbf{y}$ $\in$ $\mathbb{R}^d$, it holds
\begin{multline*}
\triangledown f^{\epsilon}(\mathbf{x}-\mathbf{y})\cdot \triangledown \left(U(\mathbf{x})-U(\mathbf{y})\right)=\frac{ \mathds{1}(\|\mathbf{x}-\mathbf{y}\|>\epsilon)}{\|\mathbf{x}-\mathbf{y}\|}(\mathbf{x}-\mathbf{y})\cdot \triangledown \left(U(\mathbf{x})-U(\mathbf{y})\right)\\
+ \frac{ \mathds{1}(\|\mathbf{x}-\mathbf{y}\|\leq \epsilon)}{\epsilon}(\mathbf{x}-\mathbf{y})\cdot \triangledown \left(U(\mathbf{x})-U(\mathbf{y})\right) \geq 0.
\end{multline*}
Therefore, the term $u_{\epsilon}[t_0,t_1]$ is easily bounded, since it is always non-positive, i.e.
$$
\liminf_{\epsilon\downarrow 0}u_{\epsilon}[t_0,t_1]\leq 0.$$
For the term $b_{\epsilon}[t_0,t_1]$, we use the global Lipschitz condition on the function $b$, together with the properties of $W_1$ Wasserstein distance and inequality~\eqref{differenza}:
\begin{multline*}
	b_{\epsilon}[t_0,t_1]\leq   \Mean\left[ \int_{t_0}^{t_1}\left\|b(X^1(s),P^1_s)-b(X^2(s),P^2_s)\right\|ds  \right]\\
	\leq L \left( \int_{t_0}^{t_1} \Mean\left[\sup_{s\in[0,t]}\|X^1(s)-X^2(s)\|\right]dt+(t_1-t_0)\rho_{[t_0,t_1]}(P^1,P^2)\right).
\end{multline*}
For estimating the term $\sigma_{\epsilon}[t_0,t_1]$, we observe that the Hessian matrix of $f^{\epsilon}$ has the following form:
\[
H_{f^{\epsilon}(x)}=\mathds{1}(\|x\|>\epsilon)\left(-\frac{1}{\|x\|^3}A+\frac{1}{\|x\|}I\right)+\mathds{1}(\|x\|\leq\epsilon)\frac{1}{\epsilon}I,
\]
where $A$ is $d\times d$ matrix such that, for all $i,j$, $A_{i,j}=x_ix_j$ and $I$ is the identity $d\times d$ matrix. Therefore, 
{\small\begin{multline*}
\sigma_{\epsilon}[t_0,t_1] \leq \\ \leq\frac{1}{2}\int_{t_0}^{t_0}\Mean\left[\frac{ \mathds{1}(\|X^1(s)-X^2(s)\|>\epsilon)}{\|X^1(s)-X^2(s)\|}Tr\left(\sigma(X^1(s),P_s^1)-\sigma(X^2(s),P_s^2) \right)^T\left(\sigma(X^1(s),P_s^1)-\sigma(X^2(s),P_s^2) \right)  \right]ds \\
+\frac{1}{2}\int_{t_0}^{t_0}\Mean\left[\frac{\mathds{1}(\|X^1(s)-X^2(s)\|\leq\epsilon)}{\epsilon}Tr\left(\sigma(X^1(s),P_s^1)-\sigma(X^2(s),P_s^2) \right)^T\left(\sigma(X^1(s),P_s^1)-\sigma(X^2(s),P_s^2) \right)  \right]ds\\
+\frac{1}{2}\int_{t_0}^{t_0}\Mean\left[\frac{ \mathds{1}(\|X^1(s)-X^2(s)\|>\epsilon)}{\|X^1(s)-X^2(s)\|^3}Tr\left(\sigma(X^1(s),P_s^1)-\sigma(X^2(s),P_s^2) \right)^T\right.\\
\left. \left((X^1(s)-X^2(s))_i(X^1(s)-X^2(s))_j\right)\left(\sigma(X^1(s),P_s^1)-\sigma(X^2(s),P_s^2) \right)  \right]ds.
\end{multline*} }
This term, due to the Lipschitz property of the diffusion coefficient $\sigma$ gives rise to a new term linear in $\Mean[\sup_{t\in[t_01,t_1]}\|X^1(t)-X^2(t)\|]$, since we have, for a certain $K\geq0$,
\[
\sigma_{\epsilon}[t_0,t_1]\leq KL\int_{t_0}^{t_1}\Mean[\sup_{s\in[t_01,t]}\|X^1(s)-X^2(s)\|]dt.
\]

In addition to the previous arguments, the treatment of the term $\Sigma_{\epsilon}[t_0,t_1]$ involves the \\ Burkholder-Davis-Gundy inequalities and the global Lipschitz condition (LD):
\begin{align*}
	\Sigma_{\epsilon}[t_0,t_1]&\leq C_1 \Mean\left[ \left(\int_{t_0}^{t_1} \left\|\left(\sigma(X^1(s),P^1_s) -\sigma(X^2(s),P^2_s)\right)^T\right.\right.\right.\\
&\qquad\qquad \qquad  \qquad\qquad \left.\left.\left.	\left(\sigma(X^1(s),P^1_s) -\sigma(X^2(s),P^2_s)\right) \right\|ds \right)^{1/2} \right] \\
	&\leq C_1 L \Mean\left[ \left( \int_{t_0}^{t_1} \sup_{s\in[t_0,t_1]}\|X^1(s)-X^2(s)\|^2 dt +(t_1-t_0)\rho_{[t_0,t_1]}(P^1,P^2)^2\right)^{1/2} \right] \\
	&\leq C_1L\sqrt{(t_1-t_0)}\left( \Mean\left[\sup_{t\in[t_0,t_1]}\|X^1(t)-X^2(t)\|\right]+\rho_{[t_0,t_1]}(P^1,P^2)\right)
\end{align*}
for some constant $C_1$ not depending on $t_0, t_1$. To bound the term $\Lambda_{[t_0,t_1]}$, we make use the properties of the process $\left\{\Lambda(t)\right\}_{t\in [0,T]}$, of the $W_1$ Wasserstein distance, as well as condition (L1) and monotone convergence theorem.
{\small\begin{align*}
	\liminf_{\epsilon\downarrow0}\Lambda_{\epsilon}[t_0,t_1]&=
\displaystyle{\Mean\left[\sup_{t\in[t_0,t_1]} \int_{t_0}^t\int_{[0,1]}\int_0^{\infty} \left\|X^1(s)+\psi(X^1(s),P^1_s,h)\mathds{1}_{u\leq\lambda(X^1(s),P^1_s)}-X^2(s)\right.\right.}\\
&\qquad \qquad \qquad \qquad \qquad \left.\left.-\psi(X^2(s),P^2_s,h)\mathds{1}_{u\leq\lambda(X^2(s),P^2_s)}\right\|	-\|X^1(s)-X^1(s)\|dsdu \nu_1(dh)\right]	\\
	 &\leq \Mean\left[\int_{t_0}^{t_1} \int_{[0,\infty)\times[0,1]} \|\psi(X^1(s^-),P^1_{s^-},h)\mathds{1}_{(0, \lambda(X^1(s^-),P^1_{s^-})]}\right. \\
	&\qquad \qquad \qquad \qquad \qquad \left. -\psi(X^2(s^-),P^2_{s^-},h)\mathds{1}_{(0, \lambda(X^2(s^-),P^2_{s^-})]}\| dsdu\nu(dh)\right]\\
	&\leq L \left(\int_{t_0}^{t_1} \Mean\left[\sup_{s\in [0,t]}\|X^1(s)-X^2(s)\|\right]dt+(t_1-t_0)\rho_{[t_0,t_1]}(P^1,P^2)\right).
\end{align*}}
Therefore,
{\begin{multline*} 
	\displaystyle{\Mean\left[\sup_{t\in[t_0,t_1]}\|X^1(t)-X^2(t)\|\right] \leq \Mean\left[\|X^1(t_0)-X^2(t_0)\|\right]}\\
	+L\left((K+1)(t_1-t_0)+C_1\sqrt{t_1-t_0}\right)\rho_{[t_0,t_1]}(P^1,P^2)\\
	\displaystyle{\qquad \qquad +  C_1 L\sqrt{(t_1-t_0)} \Mean\left[\sup_{t\in[t_0,t_1]}\|X^1(t)-X^2(t)\|\right]}\\
	\displaystyle{+ L(1+K) (t_1-t_0) \int_{t_0}^{t_1} \Mean\left[\sup_{s\in [0,t]}\|X^1(s)-X^2(s)\|\right]dt.}
\end{multline*}}
By hypothesis, $\Mean\left[\|X(0)-Y(0)\|\right] = 0$, then choose $T_0 > 0$ such that $1- C_1L\sqrt{T_0} > 0$.  
Therefore we have
\begin{multline}\label{IneqGronwall}
\displaystyle{\mathbf{E}\left[\sup_{t\in[0,T_0\wedge T]}\|X^1(t)-X^2(t)\|\right]\leq\frac{L(1+K) T_0}{1- C_1L\sqrt{T_0} }\int_0^{ T_0\wedge T}\mathbf{E}\left[\sup_{s\in [0,t]}\|X^1(s)-X^2(s)\|\right]dt}\\
+\frac{L\left((1+K) T_0+C_1\sqrt{T_0}\right)}{1- C_1L\sqrt{T_0} } \rho_{T_0}(P^1,P^2).
\end{multline}
Applying Gronwall's Lemma to \eqref{IneqGronwall}, there exists a $T_0>0$ sufficiently small such that 
\begin{equation*}
	\rho_{T_0}(Q^1,Q^2)\leq \Mean\left[\sup_{t\in[0,T_0\wedge T]}\|X^1(t)-X^2(t)\|\right] < C_{T_0}\rho_{T_0}(P^1,P^2),
\end{equation*}
for a constant $C_{T_0}<1$. Therefore, when $P^k\doteq Q^k$, this shows uniqueness of the McKean-Vlasov measure in $\mathcal{M}^{1}\left(\mathbf{D}([0,T_0],\mathbb{R}^d)\right)$. However, since $C_{T_0}$ depends only on the amplitude of the interval, the same procedure iterated  over a finite number of intervals of the type $[T_0\wedge T, 2T_0\wedge T]$, $[2T_0\wedge T,3T_0\wedge T]$, etc., yields uniqueness of the measure in $\mathcal{M}^{1}\left(\mathbf{D}([0,T],\mathbb{R}^d)\right)$.\\

The proof of existence is obtained via a Picard iteration argument, starting from \eqref{Mckean-vlasov_non_lipschitz_uniqueness}. Let $P^k\doteq Q^{k-1}$, then \eqref{Mckean-vlasov_non_lipschitz_uniqueness}  gives a sequence of laws $\{Q^k\}_{k\in\mathbb{N}}$, that is a Cauchy sequence for the metric $\rho_{T_0}$ on $\mathcal{M}^{1}\left(\mathbf{D}([0,T_0],\mathbb{R}^d)\right)$. Consequently, it is a Cauchy sequence also for a weaker Wasserstein metric based on a complete Skorohod metric, that yields existence of a solution of \eqref{Mckean-vlasov_non_lipschitz_uniqueness} on $[0,T_0\wedge T]$. Again, iterating the procedure over a finite number of intervals gives the thesis.
\end{proof}

\subsection{Propagation of chaos}

We use again the trick of the sequence of \emph{intermediate processes} $\{Y^N\}_{N\in\mathbb{N}}$, where each process $Y^N=\{Y^N(t)\}_{t\in [0,T]}$ is defined as the solution of the system \eqref{SDE_YN}.
As before,  the \emph{collateral jumps} have been absorbed by a new drift term, that by the properties of the jump rate $\lambda$ and its amplitude $\Theta$ is a globally Lipschitz drift term, that is added to $b$, giving raise to a new drift $\bar F$ that maintains condition (U) of $F$. For the proof of propagation of chaos, we apply again the procedure of Section~\ref{prop_chaos}, starting with the bound over the distance of the two particle systems \eqref{SDE_XN} and \eqref{SDE_YN}.

\begin{theorem}\label{iniziale_intermedio_NONSTANDARD_DRIFT}
Grant  
Assumption~\ref{ASS_COLL_JUMPS} and 
Assumption~\ref{ASS_NL2}. Let $X^N$ and $Y^N$ be the solution of \eqref{SDE_XN} and \eqref{SDE_YN}, respectively. We assume the two processes are driven by the same Brownian motions and Poisson random measures, and start from the same square-integrable and permutation invariant initial condition. Then, for each fixed $i \in \mathbb{N}$,
\begin{equation*} \label{dist-interm_DRIFT_NONSTANDARD}
	\lim_{N \to +\infty} \Mean\left[ \sup_{t\in[0,T]} \|X^N_i(t)-Y^N_i(t)\| \right] = 0.
\end{equation*}
\end{theorem}
\begin{proof}
By the permutation invariance of the systems we write
{\small\begin{equation*} \begin{split}
\Mean\left[ \sup_{t\in[0,T]} \|X^N_i(t)-Y^N_i(t)\| \right]  & = \frac{1}{N}\sum_{i=1}^{N}\Mean\left[\sup_{t\in[0,T]}\|X^N_i(t)-Y^N_i(t)\|\right] \\ & =\liminf_{\epsilon\downarrow0}\frac{1}{N}\sum_{i=1}^{N}\Mean\left[\sup_{t\in[0,T]}f^{\epsilon}(X^N_i(t)-Y^N_i(t))\right],
\end{split}
\end{equation*}}
where $f^{\epsilon}$ is the smooth approximation of the norm, defined in Theorem~\ref{existence_uniqueness_McKeanVlasov_NONSTANDARD_DRIFT}. Then, we use the techniques of Theorem~\ref{existence_uniqueness_McKeanVlasov_NONSTANDARD_DRIFT}, as the use of Ito's rule with the function $f^{\epsilon}$. This, together with the computations in Proposition~\ref{iniziale_intermedio} and the usual application of Gronwall Lemma iteratively over a finite number of intervals of the type $[0,T_0\wedge T],\,[T_0,2T_0\wedge T]$, etc. yields, for some constant $C_T>0$,
$$
\frac{1}{N}\sum_{i=1}^{N}\Mean\left[\sup_{t\in[0,T]}\|X^N_i(t)-Y^N_i(t)\|\right]\leq \frac{C_T}{\sqrt{N}},$$
that gives the thesis.
\end{proof}

\begin{proposition}\label{intermedio_limite_NONSTANDARD_DRIFT}
Grant Assumption~\ref{ASS_COLL_JUMPS} and Assumption~\ref{ASS_NL2}. Let $\mu_0$ be a probability measure on $\R^d$ such that $\int \|x\|^2 \mu_0(dx)<+\infty$. For $N\in \mathbb{N}$, let $Y^N$ be a solution of Eq.~\eqref{SDE_YN} in $[0,T]$. Assume that $Y^N(0) = (Y^N_1(0), \ldots, Y^N_N(0))$, $N\in \mathbb{N}$, form a sequence of  square-integrable random vectors that is $\mu_0$-chaotic in $W_1$. Let $Q$ be the law of the solution of Eq.~\eqref{SDE_limite} in $[0,T]$ with initial law $\Prb \circ X(0)^{-1} = \mu_0$. Then $Y^N$ is $Q$ chaotic in $W_1$.
\end{proposition}

\begin{proof}  We follow the steps of Proposition~\ref{intermedio_limite} to define the coupling procedure.  We fix a filtered probability space $(\Omega, \mathcal{F}, (\mathcal{F}_t)_{t\geq0},\Prb)$ with respect to the family $(B_i,\mathcal{N}^i)_{i\in\mathbb{N}}$ of independent  Brownian motions and Poisson random measures. For each $N \in \mathbb{N}$, we couple the process $Y^N$ with the process $\bar X^N=\left\{\bar X^N_i(t), \, i=1,\dots,N\right\}_{t\in[0,T]}$ defined thanks to Theorem~\ref{existence_uniqueness_McKeanVlasov_NONSTANDARD_DRIFT},  where the initial condition is $Law(\bar X^N(0))=\otimes^N \mu_{0}$ and each component $\bar X^N_i$ is a solution of \mbox{SDE} \eqref{Mckean-vlasov_non_lipschitz}. Successively, we use the techniques of the previous theorems, we iterate the computations over a finite number of time intervals to cover all $[0,T]$ and we obtain
\begin{equation*}
\displaystyle{\frac{1}{N}\sum_{i=1}^N\Mean\left[\sup_{t\in[0,T]}\|Y^N_i(t)-\bar X^N_i(t)\|\right]\stackrel{N\rightarrow\infty}{\rightarrow}0
},
\end{equation*}
that implies $Q$ chaoticity of the law of $Y^N$.
\end{proof}
\begin{corollary}
Grant Assumption~\ref{ASS_COLL_JUMPS} and Assumption~\ref{ASS_NL2}. Let $\mu_0$ be a probability measure on $\R^d$ such that $\int \|x\|^2 \mu_0(dx)<+\infty$. For $N\in \mathbb{N}$, let $X^N$ be a solution of Eq.~\eqref{SDE_XN} in $[0,T]$. Assume that $X^N(0) = (X^N_1(0), \ldots, X^N_N(0))$, $N\in \mathbb{N}$, form a sequence of  square-integrable random vectors that is $\mu_0$-chaotic in $W_1$. Let $Q$ be the law of the solution of Eq.~\eqref{SDE_limite} in $[0,T]$ with initial law $\Prb \circ X(0)^{-1} = \mu_0$. Then $X^N$ is $Q$ chaotic in $W_1$.
\end{corollary}


\section{Non-globally Lipschitz jump rate}\label{non_lip_rate}
Stochastic models in neuroscience often focus on the membrane potential of neurons and describe its spikes in terms of \mbox{SDEs} with jumps. The jump rates in those models are usually super-linear. It is therefore interesting to investigate systems where the jump coefficients are not required to be globally Lipschitz. We start by adapting the model presented in \cite{RoTo14} to a $d$-dimensional framework and a slightly more general situation allowing for jumps with random amplitude.  In order to get a tractable model with a super-linear jump rate, we are forced to make more restrictive assumptions on the other parts of the dynamics than in the previous sections. We consider a model where particles are subject to a linear attracting drift, we drop the diffusion part, and we assume the main jump to force the particles into a given compact set. Furthermore, the collateral jumps are of bounded random amplitude and do not depend on the positions of the affected jumping particle.

\subsection{Particle system}
We consider the Markov process $X^N$ solution of the following \mbox{SDE}, similar to Eq.~\eqref{SDE_XN}, {\small
\begin{multline}\label{SDE_XN_RATE_NONSTANDARD}
dX_i^N(t)=\displaystyle{-X_i^N(t)dt+\frac{1}{N}\sum_{j\neq i}\int_{[0,\infty]\times [0,1]^{\mathbb{N}}}V(h_j,h_i)\mathds{1}_{[0,\lambda(X^N_j(t)))}(u)\mathcal{N}^j(dt,du,dh)}\\
\displaystyle{\quad-\int_{[0,\infty)\times [0,1]^{\mathbb{N}}}\left(X_i^N(t)-U(h_i)\right)\mathds{1}_{[0,\lambda(X^N_i(t)))}(u)\mathcal{N}^i(dt,du,dh)}
\end{multline}} for all $i=1,\dots,N$. 
 As before, $(\mathcal{N}^i)_{i\in\mathbb{N}}$ is an independent family of  Poisson random measures $\mathcal{N}^i$, each of them with characteristic measure $l\times l\times\nu$. $l$ is the Lebesgue measure restricted to $[0,\infty)$ and $\nu$ is a symmetric probability measure on $[0,1]^{\mathbb{N}}$  such that it exists $\left(\nu_N\right)_{N\in\mathbb{N}}$, a consistent family of symmetric probability measures, each of them defined respectively on $[0,1]^N$ and coinciding with the projections of $\nu$ on $N$ coordinates. 
 {\begin{assumpt}\label{ass}
 The coefficients of the system \eqref{SDE_XN_RATE_NONSTANDARD} obey the following properties:
\begin{itemize}
\item the jump rate of each particle is a non-negative $C^1$ function of its position, $\lambda:\mathbb{R}^d\rightarrow \mathbb{R}_+ $, that is written as a sum of two functions: 
{ $$\lambda(\cdot)\doteq b(\|\cdot\|)+h(\cdot).$$ 
\begin{itemize}
\item[-] $b$ is a $C^1$, positive, non-decreasing function such that 
\begin{equation}\label{condizione_su_b}
b'(r)\leq\gamma b(r)+c
\end{equation}
for some $c>0$ and   $\displaystyle{\gamma<\frac{1}{5\Mean[\|V\|]}}$;
\item[-] $h:\mathbb{R}^d\rightarrow\mathbb{R}$ is a $C^1$ bounded function, i.e. there exists $H>0$ such that $\forall$ $x\in \mathbb{R}^d$, $\|h(x)\|\leq H$;
\end{itemize}}
\item the jump amplitudes, $V$ and $U$, are two bounded functions from respectively $[0,1]^2$ and $[0,1]$ to $\mathbb{R}^d$ (since they represents two random variables with values in some bounded subsets of $\mathbb{R}^d$, with abuse of notation we will indicate as expectations their integrals w.r.t. the measure $\nu$ ). 
\end{itemize}
The assumption $\displaystyle{\gamma<\frac{1}{5\Mean[\|V\|]}}$ will allow to obtain apriori bounds on the moments of $\lambda(X(t))$ where $X(t)$ is te solution of the corresponding McKean-Vlasov equation, see \eqref{SDE_LIMITE_NONSTANDARD_RATE}, and it is used in the proofs of next Lemmas ~\ref{bound_per_<mu,f^5>} and~\ref{bound_jumps}, which are postponed to the appendix.

\end{assumpt}}
\begin{remark} The model and the form of the function $b$ is suggested by \cite{RoTo14}. It is interesting to notice that Assumption \ref{ass} allows to consider non-globally Lipschitz functions; in particular, this covers all the cases where $b(r)$ is of the form $r^{\alpha}$, for $\alpha\geq 1$. 
We also remark that the condition on $b$ here is a little stronger than in \cite{RoTo14}, due to the coupling method (vs.  the martingale approach) in the proof, which in particular allows to identify the rate of convergence,  which is of the order $O\left(\frac{1}{\sqrt{N}}\right)$. This requires $\displaystyle{\gamma<\frac{1}{K\Mean[\|V\|]}}$ with $K=5$ rather than $K=3$, as in \cite{RoTo14}.
\end{remark}
\begin{remark}The neuronal model presented in \cite{DeGaLoPr14} and \cite{FoLo14} is similar to this one, but it has a drift toward the barycenter of the system, instead of the origin. When dealing with initial conditions with bounded support, we could adapt our computations to that case, except for the fact that the function $b$ has to be convex. In \cite{FoLo14}, the authors succeed in proving  propagation of chaos with an explicit rate (namely, the expected $\frac{1}{\sqrt{N}}$) even for weaker conditions on the initial values, by defining an ad-hoc distance based on the rate function $\lambda$ itself. 
\end{remark}
Existence and uniqueness of a non-explosive solution to the system with $N$ fixed relies on a truncation argument
on the function $\lambda$, see Appendix~\ref{app_B}.

\subsection{McKean-Vlasov equation}\label{nonlinear_SDE_e!}
This section is devoted to analyze the McKean-Vlasov equation whose law is the limit of the sequence of empirical measures corresponding to system \eqref{SDE_XN_RATE_NONSTANDARD}, that is{\small
\begin{equation}\label{SDE_LIMITE_NONSTANDARD_RATE}
dX(t)=\Mean\left[\lambda(X(t))\right]\Mean\left[V\right]dt-X(t)dt-\int_{[0,\infty)\times [0,1]^{\mathbb{N}}}\left(X(t)-U(h_1)\right)\mathds{1}_{[0,\lambda(X(t)))}(u)\mathcal{N}(dt,du,dh),
\end{equation}}
with $\mathcal{N}$ Poisson random measure with characteristic measure $l\times\nu \times l$. {Since the model that we treat is basically an extension in $d$-dimension of the model presented in \cite{RoTo14}, techniques for proving existence and uniqueness of solutions for the nonlinear Markov process \eqref{SDE_LIMITE_NONSTANDARD_RATE} are adaptations of the techniques presented in that paper. The procedure relies on a priori bounds on moments of the solution and of the expectation of $\lambda(X(t))$, we will present the main steps here, while we gather the details in Appendix~\ref{app_B}.
}

\begin{lemma}\label{Lemma_Z}
Let $f:\mathbb{R}_+\rightarrow\mathbb{R}^d$ be a locally bounded Borel function, then there exists a unique solution $\left(Z_f(t)\right)$ to the \mbox{SDE}
\begin{equation}\label{SDE_con_funct}
dZ_f(t)=-Z_f(t)dt+f(t)dt-\int_{[0,\infty)\times [0,1]^{\mathbb{N}}}\left(Z_f(t)-U(h_1)\right)\mathds{1}_{[0,\lambda(Z_f(t)))}(u)\mathcal{N}(dt,du,dh)
\end{equation}
with initial condition $x$ and coefficients satisfying Assumption~\ref{ass}. Moreover, for every pair of locally bounded Borel functions $f$ and $g$, for every $T>0$ there exists a constant $C_T>0$ such that 
\begin{equation}\label{bound_Z_u}
\Mean\left[\sup_{t\in[0,T]}\left\|Z_f(t)-Z_g(t)\right\|\right]\leq C_T\int_0^T\sup_{s\in[0,t]}\left\|f(s)-g(s)\right\|dt.
\end{equation}
\end{lemma}
A priori bounds for any solution of \eqref{SDE_LIMITE_NONSTANDARD_RATE} are necessary to perform the iteration that yields to the existence and uniqueness of the nonlinear process itself. The following lemma provides the required bounds.

\begin{lemma}\label{bound_a_priori} Suppose Assumption~\ref{ass} is satisfied.
Let $X$ be a solution of \eqref{SDE_LIMITE_NONSTANDARD_RATE} with integrable initial condition $X(0)$; then 
we have that $\sup_{t\geq0}\Mean\left[\|X(t)\|\right]<\infty$. Moreover,  for $p=1,2, 3, 4$, if $\Mean\left[\lambda^p(X(0))\right]<\infty$ then $\sup_{t\geq0}\Mean\left[\lambda^p(X(t))\right]\leq C<\infty$, where $C$ only depends on $\Mean\left[\lambda^p(X(0))\right]$ and on the parameters of equations \eqref{SDE_LIMITE_NONSTANDARD_RATE}.
\end{lemma}

The proof of this lemma is in Appendix~\ref{app_B} and it basically relies on the properties of the function $b$.
\begin{theorem}[Solution of the McKean-Vlasov equation]
Under Assumption~\ref{ass}, for any initial condition $X(0)$ with bounded support and independent of $\mathcal{N}$, there exists a unique strong solution $\{X(t)\}_{t\in[0,T]}$ for \eqref{SDE_LIMITE_NONSTANDARD_RATE}.
\end{theorem}
\begin{proof}
Fix a constant $C>0$, and consider the following Picard iteration: $Z^C_0(t)\equiv X(0)$ and
\begin{equation*}
\left\{
\begin{array}{l}
dZ^C_n(t)=-Z^C_n(t)dt+\left(\Mean\left[\lambda(Z^C_{n-1}(t))\right]\wedge C\right)\Mean\left[V\right]dt\\
\qquad\qquad\qquad-\int_{[0,\infty)\times [0,1]^{\mathbb{N}}}\left(Z^C_n(t)-U(h_1)\right)\mathds{1}_{[0,\lambda(Z^C_n(t)))}(u)\mathcal{N}(dt,du,dh),\\
Z^C_n(0)=X(0)
\end{array}
\right.
\end{equation*}
The following almost sure apriori bound is essentially obvious: 
\[
\|Z^C(t)\|\leq K+tC\Mean[\|V\|]
\]
for a suitable $K>0$ depending on the support of $X(0)$ and the range of $U(h)$. Indeed, when $\|Z^C(t)\|$ is large, the linear term $-Z^C_n(t)dt$ as well as the jumps can only decrease the norm.
From Lemma~\ref{Lemma_Z} we now that there exists a constant $C_T$ such that
\[
\Mean\left[\sup_{t\in[0,T]}\|Z^C_{n+1}(t)-Z^C_n(t)\|\right]\leq  C_T\Mean\left[\|V\|\right]\int_0^T\left\|\Mean\left[\lambda(Z^C_{n}(s))\right]-\Mean\left[\lambda(Z^C_{n-1}(s))\right]\right\|ds.
\]
Thanks to the a.s. bounds on $\|Z^C_n(t)\|$, we can exploit the local Lipschitzianity of $\lambda$ and get, for a certain constant $K_T>0$, 
\begin{align*}
\Mean\left[\sup_{t\in[0,T]}\|Z^C_{n+1}(t)-Z^C_n(t)\|\right]\leq&C_T\Mean\left[\|V\|\right]K_T\int_0^T\Mean\left[\sup_{s\in[0,t]}\|Z^C_{n}(s)-Z^C_{n-1}(s)\|\right]dt\\
\leq&\dots\leq \frac{\left(K_TC_T\Mean\left[\|V\|\right]T\right)^n}{n!}\Mean\left[\sup_{s\in[0,t]}\|Z^C_{1}(s)-Z^C_{0}(s)\|\right].
\end{align*} 
Therefore the sequence $\{Z^C_n\}_{n\in\mathbb{N}}$ is a Cauchy sequence and its limit $Z^C$ is a solution of the \mbox{SDE}
\begin{multline*}
dZ^C(t)=-Z^C(t)dt+\left(\Mean\left[\lambda(Z^C(t))\right]\wedge C\right)\Mean\left[V\right]dt\\
\qquad\qquad\qquad-\int_{[0,\infty)\times [0,1]^{\mathbb{N}}}\left(Z^C(t^-)-U(h_1)\right)\mathds{1}_{[0,\lambda(Z^C(t^-)))}(u)\mathcal{N}(dt,du,dh).
\end{multline*}
By Lemma \ref{bound_a_priori}, we can choose $C$ so that $\Mean\left[\lambda(Z^C(t))\right] \leq C$ for all $t$, so that $Z^C$ is indeed a solution of \eqref{SDE_LIMITE_NONSTANDARD_RATE}.

To prove uniqueness we can consider two solutions $Z_1$ and $Z_2$. Using the above apriori bound, \eqref{bound_Z_u} and the  Gronwall Lemma their equality follows from standard arguments.
\end{proof}

\subsection{Propagation of Chaos}\label{prop_chaos_rate}
 As in the previous sections, we introduce an intermediate process $\{Y^N(t)\}_{t\in[0,T]}$ that is the solution of a system, similar to \eqref{SDE_XN_RATE_NONSTANDARD}, that is
 {
\begin{multline}\label{SDE_YN_RATE_NONSTANDARD}
dY_i^N(t)=\displaystyle{-Y_i^N(t)dt+\frac{1}{N}\sum_{j=1}^N\Mean\left[V\right]\lambda(Y^N_j(t)) dt}\\
\displaystyle{\quad-\int_{[0,\infty)\times [0,1]^{\mathbb{N}}}\left(Y_i^N(t)-U(h_i)\right)\mathds{1}_{[0,\lambda(Y^N_i(t)))}(u)\mathcal{N}^i(dt,du,dh)}
\end{multline}} for all $i=1,\dots,N$. We indicate the empirical measure corresponding to the solution of \eqref{SDE_YN_RATE_NONSTANDARD} as $\mu^N_Y$. In order to use a coupling procedure to prove propagation of chaos, we need to set some a priori bounds on the involved quantities.  The proofs of the two following lemmas are in Appendix~\ref{app_B}, we state them here to highlight the quantities involved and the comparison with the a priori bounds for the nonlinear process \eqref{SDE_LIMITE_NONSTANDARD_RATE}.

\begin{lemma}\label{bound_per_<mu,f^5>}
For $N>0$, under Assumption~\ref{ass}, let $X^N$ and $Y^N$ be solutions, respectively, of \eqref{SDE_XN_RATE_NONSTANDARD} and \eqref{SDE_YN_RATE_NONSTANDARD}, starting from initial conditions s.t. $\Mean\left[\langle \mu^N_X(0), \lambda^4(\cdot) \rangle\right]<\infty$ and $\Mean\left[\langle \mu^N_Y(0), \lambda^4(\cdot) \rangle\right]<\infty$ . Then there exists  a certain $N_0>0$ such that it holds 
\begin{equation*}
\sup_{N\geq N_0}\sup_{t\geq0}\Mean\left[\langle \mu^N_X(t), \lambda^4(\cdot) \rangle\right]<\infty \qquad \qquad \text{ and } \qquad \qquad \sup_{N\geq N_0}\sup_{t\geq0}\Mean\left[\langle \mu^N_Y(t), \lambda^4(\cdot)\rangle\right]<\infty.
\end{equation*}
\end{lemma}

Lemma~\ref{bound_per_<mu,f^5>} is crucial for proving that the number of jumps of the system in a compact time interval is proportional to $N$ with probability increasing with $N$. This bound is stated in the following lemma.

\begin{lemma}[Bound on the number of jumps]\label{bound_jumps} Assume that Assumption~\ref{ass} is satisfied, that, for any $N>0$, $X^N$ and $Y^N$ are solutions, respectively, of \eqref{SDE_XN_RATE_NONSTANDARD} and \eqref{SDE_YN_RATE_NONSTANDARD}, starting from  initial conditions that are $\mu_0$-chaotic. Here $\mu_0$ is a probability measure on $\mathbb{R}^d$ s.t. $\Mean_{\mu_0} \left[ \lambda^3(X) \right]<\infty$.  Then, for any $T>0$, there exists a positive constant $H_T$ and a natural number $N_0>0$ such that, for certain positive constants $K_T$  and $\tilde K_T$ 
$$
\begin{array}{ccc}
\displaystyle{
\Prb\left(\frac{C_N(T)}{N}\geq H_T\right)\leq \frac{K_T}{N}} & \text{ and } & \displaystyle{
\Prb\left(\int_0^T\langle\mu^N_Y(s),\lambda \rangle ds\geq H_T\right)\leq \frac{\tilde K_T}{N}},
\end{array}
$$
 for all $N>N_0$. Here $C_N(T)$ is the number of jumps performed by system \eqref{SDE_XN_RATE_NONSTANDARD} up to time $T$. 
\end{lemma}
The bounds on the number of \emph{collateral} jumps and of the corresponding drift in a compact time interval plays a role in the proof of propagation of chaos, since they let us exploit the local Lipschitzianity of the function $\lambda$ when we start from initial conditions with bounded support. The proofs of these lemmas involve the form of the function $\lambda$ and they are in Appendix~\ref{app_B}. In the following we state and prove the result on propagation of chaos and also in this case, the simultaneous jumps result in a rate of the order $\frac{1}{\sqrt{N}}$. As in the previous sections, we start with the comparison between the particle system $X^N$ and the intermediate system $Y^N$.

\begin{theorem}\label{iniziale_intermedio_rate}
Let Assumptions~\ref{ass} and \ref{ASS_COLL_JUMPS}  be satisfied and let $X^N$ and $Y^N$ be the solution, respectively,  of \eqref{SDE_XN_RATE_NONSTANDARD} and \eqref{SDE_YN_RATE_NONSTANDARD} with initial conditions $X^N(0)=Y^N(0)$ a.s. that are $\mu_0$-chaotic, with $\mu_0$  probability measure on $\mathbb{R}^d$ with compact support.
We assume the two processes are driven by the same Poisson random measures, and start from the same permutation invariant initial condition with compact support. Then, for each fixed $i \in \mathbb{N}$,
\begin{equation*} 
	\lim_{N \to +\infty} \Mean\left[ \sup_{t\in[0,T]} \|X^N_i(t)-Y^N_i(t)\| \right] = 0.
\end{equation*}

\end{theorem}
\begin{proof}
As in previous sections, by permutation invariance of the initial conditions and of the dynamics, we have
$$
\Mean\left[ \sup_{t\in[0,T]} \|X^N_i(t)-Y^N_i(t)\| \right]=\frac{1}{N}\sum_{i=1}^N\Mean\left[ \sup_{t\in[0,T]} \|X^N_i(t)-Y^N_i(t)\| \right].
$$

Let us start with {\small
\begin{equation*}
\Mean\left[\sup_{t\in[0,T]}\|X_i^N(t)-Y_i^N(t)\|\right]\leq \Mean\left[\int_0^T \left\|X_i^N(t)-Y_i^N(t) \right\|dt\right]
+V_{X_i^N,Y_i^N}(T)+U_{X^N_i,Y^N_i}(T),
\end{equation*}  }
where, for simplicity, we have set: 
\begin{itemize}
\item[]{\small\begin{multline*}
V_{X_i^N,Y_i^N}(T)\colon = \Mean\left[\sup_{t\in[0,T]}\left\|\frac{\Mean[V]}{N}\sum_{i=1}^N\int_0^t\lambda(X^N_j(s))-\lambda(Y^N_j(s))ds\right.\right.\\
\left.\left.+\frac{1}{N}\sum_{j\neq i}\int_0^t\int_{[0,1]^{\mathbb{N}}}\int_0^{\infty}V(h_i,h_j)\mathds{1}_{[0,\lambda(X^N_j(s))}(u)\tilde{\mathcal{N}}^j(ds,du,dh)\right\|\right];
\end{multline*}}
\item[]{\small\begin{multline*}U_{X^N_i,Y^N_i}(T)\colon = \Mean\left[\sup_{t\in[0,T]}\left\|-\int_0^t\int_{[0,1]^{\mathbb{N}}}\int_0^{\infty}(X^N_i(s)-U(h_i))\mathds{1}_{[0,\lambda(X^N_i(s)))}(u)\right.\right.\\
\left.\left.
-(Y^N_i(s)-U(h_i))\mathds{1}_{[0,\lambda(Y^N_i(s)))}(u)\mathcal{N}^i(ds,du,dh)\right\|\right].
\end{multline*}}
\end{itemize}
With the notation of Lemma~\ref{bound_jumps}, we consider the positive constant $H_T$ and the event 
$$
E_N\colon = \left\{ \frac{C_N(T)}{N}\leq H_T
\right\}\cap \left\{ \int_0^T\langle\mu^N_Y(s),\lambda \rangle ds\leq H_T \right\},$$
such that $\Prb\left(E^c_N\right)\rightarrow 0$ for $N\rightarrow \infty.$ Obviously, under the event $E_N$, for all $i=1,\dots,N$, the quantities $\sup_{t\in[0,T]}\lambda(X^N_i(t))$ and $\sup_{t\in[0,T]}\lambda(Y^N_i(t))$ are uniformly bounded and we can exploit local Lipschitzianity of $\lambda$ (we will indicate its Lipschitz constant as $L_{H_T}$). Thus, we bound the first terms in $V_{X^N_i,Y^N_i}(T)$ in the following way:
\begin{multline*}
\Mean\left[\sup_{t\in[0,T]}\left\|\frac{\Mean[V]}{N}\sum_{j=1}^N\int_0^t\lambda(X^N_j(s))-\lambda(Y^N_j(s))ds\right\|\right]\\
\leq \frac{\Mean[\|V\|]}{N} \sum_{j=1}^N  \Mean \left[\left(\int_0^TL_{H_T}\|X^N_j(s)-Y^N_j(s)\|ds\right)\mathds{1}_{E_N}\right]\\
+\Mean[\|V\|] \Mean \left[\left(\int_0^T\frac{1}{N}\sum_{j=1}^N|\lambda(X^N_j(s))|+|\lambda(Y^N_j(s))| ds\right)\mathds{1}_{E^C_N}\right]\\
\leq L_{H_T} \Mean[\|V\|] \int_0^T\frac{1}{N}\sum_{j=1}^N\Mean\left[\sup_{s\in[0,t]}\|X_j^N(s)-Y_j^N(s)\|\right]dt\\
+\int_0^T\frac{\Mean[\|V\|]}{N}\sum_{j=1}^N\sqrt{\Prb(E^C_N)}\sqrt{\Mean\left[|\lambda(X^N_j(s))|^2\right]}ds
+\int_0^T\frac{\Mean[\|V\|]}{N}\sum_{j=1}^N\sqrt{\Prb(E^C_N)}\sqrt{\Mean\left[|\lambda(Y^N_j(s))|^2\right]}ds\\\leq 
L_{H_T} \Mean[\|V\|] \int_0^T\frac{1}{N}\sum_{j=1}^N\Mean\left[\sup_{s\in[0,t]}\|X_j^N(s)-Y_j^N(s)\|\right]dt\\
+\int_0^T\Mean[\|V\|]\sqrt{\Prb(E^C_N)}\sqrt{\Mean\left[\langle\mu^N_X(s),|\lambda(\cdot)|^2\rangle\right]}ds+\int_0^T\Mean[\|V\|]\sqrt{\Prb(E^C_N)}\sqrt{\Mean\left[\langle\mu^N_Y(s),|\lambda(\cdot)|^2\rangle\right]}ds.
\end{multline*}
By Lemma~\ref{bound_per_<mu,f^5>} there exists $N_0>0$ such that for all $N>N_0$ $\sup_{t\geq0}\Mean\left[\langle\mu^N_X(s),|\lambda(\cdot)|^2\rangle\right]$ and  \\
$\sup_{t\geq0}\Mean\left[\langle\mu^N_Y(s),|\lambda(\cdot)|^2\rangle\right]$ are bounded. By Lemma~\ref{bound_jumps}, there exists a constant $K_T\geq0$ such that $\Prb(E^C_N)\leq\frac{K_T}{N}$. The second term in $V^N_{X^N_i,Y^N_i}(T)$ is bounded using Burkholder-Davis-Gundy inequality, the orthogonality of the martingales $\{\tilde {\mathcal{N}}^j\}_{j\in\mathbb{N}}$ and Lemma~\ref{bound_per_<mu,f^5>}. 
\begin{multline*}
\Mean\left[\sup_{t\in[0,T]}\left\|\frac{1}{N}\sum_{j\neq1}^N\int_0^t\int_0^t\int_{[0,1]^{\mathbb{N}}}\int_0^{\infty}V(h_i,h_j)\mathds{1}_{(0,\lambda(X^N_j(s))](u)}\tilde{\mathcal{N}}^j(ds,du,dh)\right\|\right]\\
\leq \frac{M}{N} \Mean\left[\left(\sum_{j\neq i}^N\int_0^T \Mean[\|V\|^2]\lambda(X^N_j(s))ds\right)^{1/2}\right]\leq \sqrt{\frac{\Mean[\|V\|^2]}{N}}\Mean\left[\left(\int_0^T\langle\mu^N_X(t),\lambda(\cdot)\rangle dt\right)^{1/2}\right].
\end{multline*}
Therefore we get that there exists two constants $C_T$ and $K_T$ such that, for all $N>N_0$,
\begin{equation*}
V^N_{X^N_i,Y^N_i}(T)\leq C_T\int_0^T\Mean\left[\sup_{s\in[0,t]}\|X^N(s)-Y^N(s)\|^2\right] dt+\frac{K_T}{\sqrt{N}}.
\end{equation*}
With a similar argument, we get a bound of the same type for $U_{X^N_i,Y^N_i}(T)$. 
{\begin{multline*}
\frac{1}{N}\sum_{i=1}^NU_{X^N_i,Y^N_i}(T)\leq  C_T\int_0^T\frac{1}{N}\sum_{i=1}^N\|X^N_i(t)-Y^N_i(t)\| dt + \Mean\left[\mathds{1}_{E^C_{N}}\int_0^T\frac{1}{N}\sum_{i=1}^N\|X_i^N(t)\|\lambda(X^N_i(t))dt\right]\\
+ \Mean\left[\mathds{1}_{E^C_{N}}\int_0^T\frac{1}{N}\sum_{i=1}^N\|Y_i^N(t)\|\lambda(Y^N_i(t))dt\right]+ \Mean[\|U\|]\Mean\left[\mathds{1}_{E^C_{N}}\int_0^T\frac{1}{N}\sum_{i=1}^N\lambda(X^N_i(t))dt\right]\\
+  \Mean[\|U\|]\Mean\left[\mathds{1}_{E^C_{N}}\int_0^T\frac{1}{N}\sum_{i=1}^N\lambda(Y^N_i(t))dt\right].
\end{multline*}}
As before, we wish to get a bound of the order $\displaystyle{O\left(\frac{1}{\sqrt{N}}\right)}$ for the last terms. We do that by means of Cauchy-Schwartz inequality, Lemma~\ref{bound_per_<mu,f^5>} and Lemma~\ref{bound_jumps}. We also exploit that, by definition of  $\lambda$ , it holds $\|x\|\leq B\lambda(x)+c$ for a positive constant $B$ and a constant $c$. Take, for instance, the second term of the right-hand side, it holds
\begin{multline*}
\Mean\left[\mathds{1}_{E^C_{N}}\int_0^T\frac{1}{N}\sum_{i=1}^N\|X_i^N(t)\|\lambda(X^N_i(t))dt\right]\leq \int_0^T\sqrt{\Prb(E^C_N)}\sqrt{\Mean\left[\left(\frac{1}{N}\sum_{i=1}^N\|X^N_i(s)\|\lambda(X^N_i(s)) \right)^2\right]} ds\\
\leq T \sqrt{\Prb(E^C_N)}\sqrt{ \Mean\left[\sup_{t\in [0,T]}\langle \mu^N_{X}(t),\|\cdot\|^2\rangle\langle \mu^N_{X}(t),\lambda(\cdot)^2\rangle\right]}\leq  T \sqrt{\Prb(E^C_N)}\sqrt{ \Mean\left[\sup_{t\in [0,T]}\langle \mu^N_{X}(t),\lambda(\cdot)^4\rangle\right]}. 
\end{multline*} 
The same holds for the remaining right-hand side terms. Thus, there exists two constants $\tilde C_T$ and $\tilde K_T$  and a $N_0>0$, such that for all $N>N_0$ it holds
 \begin{equation*}
\frac{1}{N}\sum_{i=1}^NU_{X^N_i,Y^N_i}(T)\leq \tilde C_T\int_0^T\Mean\left[\sup_{s\in[0,t]}\|X^N(s)-Y^N(s)\|\right] dt+\frac{\tilde K_T}{\sqrt{N}}.
\end{equation*}

Then, there exist two constants, that with abuse of notation we will indicate as $C_T$ and $K_T$, depending only on $T$, and $N_0>$ such that, for all $N>N_0$ it holds
\begin{equation*}
\Mean\left[\sup_{t\in[0,T]}\|X^N(t)-Y^N(t)\|\right]\leq C_T \int_0^T \Mean\left[\sup_{s\in[0,t]}\|X^N(s)-Y^N(s)\|\right]dt +\frac{K_T}{\sqrt{N}}.
\end{equation*}  
By applying Gronwall lemma we get the thesis.
\end{proof}

\begin{theorem}[Propagation of Chaos for $Y^N$]

Grant Assumptions~\ref{ass} and \ref{ASS_COLL_JUMPS}. 
Let $\mu$ be a probability measure on $\R^d$  with compact support. For $N\in \mathbb{N}$, let $Y^N$ be a solution of Eq.~\eqref{SDE_YN_RATE_NONSTANDARD} in $[0,T]$. Assume that $Y^N(0) = (Y^N_1(0), \ldots, Y^N_N(0))$, $N\in \mathbb{N}$, form a sequence of  integrable random vectors that is $\mu$-chaotic in $W_1$. Let $Q$ be the law of the solution of Eq.~\eqref{SDE_LIMITE_NONSTANDARD_RATE} in $[0,T]$ with initial law $\Prb \circ X(0)^{-1} = \mu$. Then $Y^N$ is $Q$ chaotic in $W_1$.

\end{theorem}
The proof of this theorem is a combination of the computations done for proving Theorem~\ref{iniziale_intermedio_rate} and the coupling techniques for propagation of chaos used in the previous sections. Again, this implies propagation of chaos for $X^N$.
\begin{corollary}[Propagation of Chaos for $X^N$]

Grant Assumptions~\ref{ass} and \ref{ASS_COLL_JUMPS}. 
Let $\mu$ be a probability measure on $\R^d$  with compact support. For $N\in \mathbb{N}$, let $X^N$ be a solution of Eq.~\eqref{SDE_XN_RATE_NONSTANDARD} in $[0,T]$. Assume that $X^N(0) = (X^N_1(0), \ldots, X^N_N(0))$, $N\in \mathbb{N}$, form a sequence of  integrable random vectors that is $\mu$-chaotic in $W_1$. Let $Q$ be the law of the solution of Eq.~\eqref{SDE_LIMITE_NONSTANDARD_RATE} in $[0,T]$ with initial law $\Prb \circ X(0)^{-1} = \mu$. Then $X^N$ is $Q$ chaotic in $W_1$.

\end{corollary}

\appendix
\section{Appendix for Section~\ref{non_lip_drift}}\label{app_A}

We gather in this subsection some lemmas useful to the analysis of the nonlinear stochastic differential equation in the case of the non-globally Lipschitz drift condition stated in (U). These two lemmas involve standard results and are both used for the Picard-iteration procedure in the proof of Theorem~\ref{existence_uniqueness_McKeanVlasov_NONSTANDARD_DRIFT}.
We could not find a general result on \mbox{SDE} with unbounded jump's rate and a non globally Lipschitz condition on the drift coefficient, so we prove it here. It is an application of classical approach, see for example Ikeda Watanabe \cite{IkWa14}, together with the trick used in the proof of Theorem~\ref{existence_uniqueness_McKeanVlasov_NONSTANDARD_DRIFT}.

\begin{lemma}\label{lemma_SDE_parametrizzata}
Consider the \mbox{SDE} parametrized by a measure $\alpha$ $\in$ $\mathcal{M}^{1}(\mathbf{D}([0,T],\mathbb{R}^d))$
\begin{multline}\label{SDE_parametrizzata}
	dX(t)=F(X(t),\alpha_t)dt + \sigma(X(t),\alpha_t)dB_t\\
	+ \int_{[0,\infty)\times[0,1]^{\mathbb{N}}} \psi(X(t^-),\alpha_{t^-},h_1)\mathds{1}_{(0,\lambda(X(t^-),\alpha_{t^-})]}(u) \mathcal{N}(dt,du,dh).
\end{multline}
If the coefficients satisfy Assumption~\ref{ASS_NL2}, then for every {\small$\alpha$ $\in$ $\mathcal{M}^1(\mathbf{D}([0,T],\mathbb{R}^d))$} and every square-integrable initial condition, there exists a unique strong solution to Eq.~\eqref{SDE_parametrizzata}.
\end{lemma}
\begin{proof}
First let $X^1$ and $X^2$ be two integrable stochastic processes on $[0,T]$  with values in $\mathbb{R}^d$. We define the map that associates the law of $X^k$ to the law of the solution of 
\begin{multline}\label{ODE_lineare_drift}
	dY^k(t)=F(Y^k(t),\alpha_t)dt + \sigma(Y^k(t),\alpha_t)dB_t\\
	+ \int_{[0,\infty)\times[0,1]^{\mathbb{N}}} \psi(X^k(t^-),\alpha_{t^-},h_1)\mathds{1}_{(0,\lambda(X^k(t^-),\alpha_{t^-})]}(u) \mathcal{N}(dt,du,dh),
\end{multline}
that is well-defined for Lemma~\ref{lemma_SDE_parametrizzata_2}. With the same computation of the proof of Theorem~\ref{existence_uniqueness_McKeanVlasov_NONSTANDARD_DRIFT}, we get that, for a small enough $T_0>0$, there exists a constant $C_{T_0}<1$ such that $$
\Mean\left[\sup_{t\in [0,T_0]}\|Y^1(t)-Y^2(t)\|\right]\leq C_{T_0}\Mean\left[\sup_{t\in [0,T_0]}\|X^1(t)-X^2(t)\|\right].
$$
This shows \emph{pathwise uniqueness} for solution of \eqref{SDE_parametrizzata}. By means of \eqref{ODE_lineare_drift}, we define a Picard iteration argument that gives a sequence of laws $\left\{Q^n\right\}_{n\in\mathbb{N}}$ on $\mathbf{D}([0,T],\mathbb{R}^d)$. Again, there exists a $T_0>0$ small enough such that $\left\{Q^n\right\}_{n\in\mathbb{N}}$ is a Cauchy sequence for $\rho_{T_0}$ and hence for a weaker but complete Wasserstein metric on $\mathcal{M}^1(\mathbf{D}([0,T_0],\mathbb{R}^d))$. Iterating the procedure over a finite number of time intervals, to cover $[0,T]$, yields the thesis.

The integrability property is proved as in the proof of Lemma~\ref{lemma_SDE_parametrizzata_2}.
\end{proof}

 \begin{remark} Notice that, in the proof of Lemma~\ref{lemma_SDE_parametrizzata}, we need to define the map by means of \eqref{ODE_lineare_drift} and not to straightly substituting $X^k$ in the whole right-hand side of \eqref{SDE_parametrizzata}. In fact, we need to control the jumps by means of a known process, but at the same time, we need to have the same variable as argument of the drift coefficient to exploit the convexity of the potential function $U$. 
 \end{remark}

\begin{lemma}\label{lemma_SDE_parametrizzata_2}
Consider the \mbox{SDE} parametrized by two measures $\alpha$ and $\beta$ $\in$ $\mathcal{M}(\mathbf{D}([0,T],\mathbb{R}^d))$
\begin{multline}\label{SDE_parametrizzata_2}
	dX(t)=F(X(t),\alpha_t)dt + \sigma(X(t^-),\alpha_t)dB_t\\
	+ \int_{[0,\infty)\times[0,1]^{\mathbb{N}}} \psi(Y(t^-),\alpha_{t^-},h_1)\mathds{1}_{(0,\lambda(Y(t^-),\alpha_{t^-})]}(u) \mathcal{N}(dt,du,dh),
\end{multline}
with $Law(Y)=\beta$.
If the coefficients satisfy Assumption~\ref{ASS_NL2}, then for every square-integrable initial condition and every {\small$\alpha$ and $\beta$ $\in$ $\mathcal{M}^1(\mathbf{D}([0,T],\mathbb{R}^d))$} , there exists a unique strong solution to Eq.~\eqref{SDE_parametrizzata_2}.\\

Moreover, let $\mu\doteq Law((X(t))_{t\in[0,T]})$ be the law of the solution of \eqref{SDE_parametrizzata_2} starting from the square-integrable initial condition $X(0)$ $\mu_0$-distributed, then $\mu$~$\in$~$\mathcal{M}^1(\mathbf{D}([0,T],\mathbb{R}^d))$. 
\end{lemma}
\begin{proof}
Let $B$ be an $(\mathcal{F}_t)$-brownian motion, $p$ be a $(\mathcal{F}_t)$-stationary Poisson point process with characteristic measure $ l \times\nu$ and $\xi$ be a $\mathcal{F}_0$-measurable square-integrable r.v.. Let $D\doteq\{s\in D_p \text{ s.t. } p(s) \, \in\, \bar U_s=(0,\lambda(Y(s^-),\alpha_{s^-})]\times[0,1]\times[0,1]\times\dots  \}$. Let us call $\sigma_1<\sigma_2<\dots$ the elements of $D$. Each $\sigma_n$ is an $\mathcal{F}_t$~-stopping time  and $\lim_{n\rightarrow\infty}\sigma_n=\infty$ a.s.. Indeed, for every $T>0$ and for a fixed $n\in\mathbb{N}^*$,
{\small\begin{equation*}
\Prb(\sigma_n\leq T)=\Prb\left(\int_0^T\int_{[0,\infty)\times[0,1]^{\mathbb{N}}} \mathds{1}_{(0,\lambda(Y(t^-),\alpha_{t^-})]}(u) \mathcal{N}(du,dh,dt)\geq n\right)\leq \frac{\Mean\left[\lambda(Y(T),\alpha_T)\right]}{n}\leq\frac{C_T}{n},
\end{equation*} }
for a certain constant $C_T$. By Lemma~\ref{lemma_jumps}, we get the claim.  Then we start by showing $\exists !$ of a solution for \eqref{SDE_parametrizzata_2} on $[0,\sigma_1]$. Consider the equation
\begin{equation}\label{nojumps_2}
Z(t)=X(0)+\int_0^tF(Z(s),\alpha_s)ds+\int_0^t\sigma(Z(s^-),\alpha_s)dB_s.
\end{equation}
Existence and uniqueness of a strong solution for \eqref{nojumps_2} are ensured by the classical Hasminskii's test for non-explosion (see e.g. \cite{mckean1969} with the Lyapunov function $V(z) = \|z\|^2$). The test's conditions are guaranteed by the inequality
\begin{equation}\label{cond_non_lipschitz_2}
\sup_{\alpha\in\mathcal{M}^1(\mathbb{R}^d)}z\cdot F(z,\alpha)+tr(\sigma(z,\alpha)\sigma^T(z,\alpha))
\leq C(1+\|z\|^2),
\end{equation}
for some $C>0$, for all $z$ $\in$ $\mathbb{R}^d$.
 Indeed, fix an $\alpha$ $\in$ $\mathcal{M}^1(\mathbb{R}^d)$. Then, under $(U)$ from Assumption~\ref{ASS_NL2}, we have
 \[
 z\cdot F(z,\alpha)=-(z-\underline{0})\cdot (\triangledown U(z)-\triangledown U(\underline{0})) + z \cdot  \triangledown U(\underline{0}) + z\cdot b(z,\alpha)  \leq C\left(\|z\|^2 +1 \right),
\]
due to the convexity of $U$ and the linear growth of $b$. A similar bound is obtained for the second summand in the 
l.h.s of \eqref{cond_non_lipschitz_2}, which has uniform quadratic growth in the $z$ variable. Then, for every integrable initial condition, there exists a unique strong solution to \eqref{nojumps_2}. Let $\pi_1$ be the projection defined as 
 \begin{equation*}
 \begin{array}{rccc}
 \pi_1:&[0,1]^{\mathbb{N}}\times [0,\infty)&\mapsto & [0,1]\\
 &(h,u)&\rightarrow& h_1,
 \end{array}
 \end{equation*}
 we define 
 \begin{equation}\label{betweenjumps_2}
 X_1(t)=\left\{
 \begin{array}{ll}
 Z^1(t)& t\,\in\,[0,\sigma_1),\\
 Z^1(\sigma_1^-)+\psi(Z^1(\sigma_1^-),\alpha(\sigma_1^-),\pi_1\circ p(\sigma_1))& t=\sigma_1,
 \end{array}
 \right.
 \end{equation}
 where $\{Z^1(t)\}_{t\geq0}$ is solution of \eqref{nojumps_2} with initial condition $Z^1(0)=\xi$ a.s.. We see that $X^1(t)$ is solution of \eqref{SDE_parametrizzata_2}  for $t$ $\in$ $[0,\sigma_1]$. We iterate the procedure by setting $\bar \xi\doteq X_1(\sigma_1)$, $\bar B\doteq(B(t+\sigma_1)-B(\sigma_1))_{t\geq0}$ and $\bar p\doteq(p(t+\sigma_1))_{t\geq0}$. We define $\bar X_1(t)$ for $t$ $\in$ $[0,\bar\sigma_1]$ as we did for $X_1(t)$ in \eqref{betweenjumps_2}, where $\bar \sigma_1$ is the smallest time such that $\bar p_s$ belongs to $\bar U_{\sigma_1+s}$ and coincides with $\sigma_2-\sigma_1$. We define 
  \begin{equation*}
 X_2(t)=\left\{
 \begin{array}{ll}
 X_1(t)& t\,\in\,[0,\sigma_1],\\
 \bar X_1(t-\sigma_1)&  t\,\in\,[\sigma_1,\sigma_2].
 \end{array}
 \right.
 \end{equation*}
 Clearly $X_2$ is solution of \eqref{SDE_parametrizzata_2}  for $t$ $\in$ $[0,\sigma_2]$. Since $\lim_{n\rightarrow\infty}\sigma_n=\infty$ a.s., we can iterate this procedure to cover the entire time interval $[0,T]$.

To prove that the law $\mu$ of a solution of \eqref{SDE_parametrizzata_2} belongs to $\mathcal{M}^1(\mathbf{D}([0,T],\mathbb{R}^d))$, we will show that there exists a filtered probability space $(\Omega,\Prb,(\mathcal{F}_t),\mathcal{F})$, with a $\mathcal{F}_t$-Brownian motion $B$, an adapted $\mathcal{F}_t$ Poisson random measure $\mathcal{N}$ with characteristic measure $l\times l\times \nu$ and a $\mathcal{F}_0$-measurable initial condition $X(0)$ $\mu_0$-distributed such that $\Mean\left[\sup_{t\in[0,T]}\|X(t)\|\right]<\infty$. 
We consider the process $X(t)$, for all $t\geq0$, solution of \eqref{SDE_parametrizzata_2}. Now, we use the trick of applying Ito's rule to the smooth approximation $f^{\epsilon}$ of $\|\cdot\|$ and taking the limit for $\epsilon\downarrow0$, to exploit the properties of the potential function $U$. For the details of the approach, see the proof of Theorem~\ref{existence_uniqueness_McKeanVlasov_NONSTANDARD_DRIFT}.
 Then, for the properties of coefficients and quantities involved, there exist three positive constants $D_1$, $D_2$ and $D_3$ s.t.
 \[
\Mean\left[\sup_{t\in[0,T]}\| X(t)\|\right]\leq \Mean\left[\|X(0)\|\right]+D_1 T+D_2T\Mean\left[\sup_{t\in[0,T]} \|Y(t)\|\right]+D_1\int_0^T\Mean\left[\sup_{s\in[0,t]}\| X(s)\|\right]dt.
\] 
We apply Gronwall Lemma and we get the desired bound. 
 \end{proof}
\begin{lemma}\label{lemma_jumps}
Let $\{\sigma_n\}_{n\in\mathbb{N}^*}$ be a sequence of strictly increasing stopping times. If, for all $T>0$, there exists a constant $C_T\geq0$ such that
$$
\mathbf{P}(\sigma_n\leq n)\leq\frac{C_T}{n},$$
then $\lim_{n\rightarrow\infty}\sigma_n=\infty$ a.s..
\end{lemma}
\begin{proof} We start by proving that, for all $T>0$, there exists a measurable set $\Lambda_T$ with probability one, such that for all $\omega$ $\in$ $\Lambda_T$, there exists $n_0(\omega,T)$ and for all $n\geq n_0(\omega,T)$ it holds $\sigma_n(w)>T$. \\
Let $A_n\doteq\left\{\sigma_{n^2}\leq T\right\}$ and $\displaystyle{A\doteq \bigcap_{n=1}^{\infty}\bigcup_{i=n}^{\infty}A_i}$, therefore we have 
$$
\displaystyle{\sum_{n=1}^{\infty}\mathbf{P}(A_n)\leq \sum_{n=1}^{\infty}\frac{C_T}{n^2}<\infty}
$$
and for Borel Cantelli $\mathbf{P}(A)=0$. Let $\Lambda_T\doteq A^{C}$, then it has probability  one and for all $\omega\in\Lambda_T$ there exists $\bar n_0(\omega,T)$ such that for all $n\geq \bar n_0(\omega,T)$ we have $\sigma_{n^2}>T$. Since the $\sigma_n$ are increasing, we have the claim that there exists $n_0(\omega,T)$ such that for all $n\geq n_0(\omega, T)$, $\sigma_n>T$.  \\
Now, let $\displaystyle{\tilde{\Lambda}\doteq\bigcap_{T\in\mathbb{N}}}\Lambda_T$, then $\mathbf{P}(\tilde{\Lambda})=1$ and for all $\omega$ $\in$ $\tilde{\Lambda}$ for all $T\geq0$ there exists $n_0(\omega,T)$ s.t. for all $n\geq n_0(\omega,T)$ then $\sigma_n(\omega)>T$. This implies $\sigma_n\nearrow\infty$ a.s..
\end{proof}


\section{Appendix for Section~\ref{non_lip_rate}}\label{app_B}
We collect here auxiliary lemmas and proofs for Section~\ref{non_lip_rate}. First, Lemma~\ref{exists_N_particle} concerns existence and uniqueness of solutions for the particle system \eqref{SDE_XN_RATE_NONSTANDARD} under Assumption~\ref{ass}. Notice that the same result holds also for the intermediate particle system \eqref{SDE_YN_RATE_NONSTANDARD}. Then, thanks to two technical lemmas, we give the proof of Lemma~\ref{Lemma_Z}, crucial for the existence and uniqueness of solution of the nonlinear process \eqref{SDE_LIMITE_NONSTANDARD_RATE}. Finally, we give the proofs of Lemma~\ref{bound_per_<mu,f^5>} and Lemma~\ref{bound_jumps}, that we use in the propagation of chaos section. Notice that, the key ingredient here is represented by the fact that all the main jumps of the processes are such that they make the process go back inside a compact set (the support of $U$). To exploit that, we need to apply Ito's rule for a process with jumps (notice that here we do not have a diffusion term). Since all the functions of interest ($\|\cdot\|$ and $\lambda(\cdot)$) have singularities in the origin, we use a smooth approximation of the norm $\|\cdot\|$. As in the proof of Theorem~\ref{existence_uniqueness_McKeanVlasov_NONSTANDARD_DRIFT}, for all $\epsilon>0$, we define 
 $$f^{\epsilon}(x)\doteq \|x\|\mathds{1}(\|x\|>\epsilon)+\left(\frac{\|x\|^2}{2\epsilon}+\frac{\epsilon}{2}\right)\mathds{1}(\|x\|\leq \epsilon).$$ 

\begin{lemma}\label{exists_N_particle} Under Assumption~\ref{ass}, for every integrable initial condition $X^N(0)$ $\in$ $\mathbb{R}^{d\times N}$, the \mbox{SDE} \eqref{SDE_XN_RATE_NONSTANDARD} admits a unique solution. 
\end{lemma}
\begin{proof}

The main issue is represented by the fact that the function $\lambda$ is unbounded and not globally Lipschitz continuous, indeed when $\lambda$ is bounded existence and uniqueness of solutions for \eqref{SDE_XN_RATE_NONSTANDARD} are consequences of standard results, see \cite{IkWa14}. Therefore, let us consider the truncate function $\lambda^K\doteq \lambda \wedge K$, for $K$ $\in $ $\mathbb{N}$, and the solution $X^{N,K}(t)$ of \eqref{SDE_XN_RATE_NONSTANDARD} with the function $\lambda^K$ instead of $\lambda$. This solution exists and it is unique for all $t$ $\in$ $[0,T]$, moreover, by pathwise uniqueness, it holds
$X^{N,K}(t)=X^{N,K+1}(t)$ for all $t$ $\in$ $\tau^K$, where $\tau^K\doteq\inf\left\{t\, / \|X^{N,K}(t)\|\geq K\right\}$. Therefore $\tau^{K}\leq\tau^{K+1}$ a.s. and there exists a pathwise unique solution $X(t)$ to \eqref{SDE_XN_RATE_NONSTANDARD}, defined for all $t$ $\in$ $[0,\tau)$, where $\tau\doteq\sup_{K\in\mathbb{N}}\tau^K$. We are left to prove that $\mathbf{P}(\tau>T)=1$.

Let us fix $i$~$\in$~$\{1,\dots,N\}$ and $\epsilon>0$. By computing $f^{\epsilon}(X^N_i(t))$ by means of Ito's formula, we get
\begin{multline*}
f^{\epsilon}(X^N_i(t))
%
\leq f^{\epsilon}(X^N_i(0))
+\frac{1}{N}\sum_{j=1}^N\int_0^t\int_{[0,1]^{\mathbb{N}}}\int_0^{\infty}f^{\epsilon}\left(V(h_j,h_i)\right)\mathds{1}_{(0,\lambda(X^N_j(s))]}(u)\mathcal{N}^j(ds,du,dh)\\
+\int_0^t\int_{[0,1]^{\mathbb{N}}}\int_0^{\infty}\left(f^{\epsilon}\left(U(h_i)\right)-f^{\epsilon}\left(X^N_i(s)\right)\right)\mathds{1}_{(0,\lambda(X^N_i(s))]}(u)\mathcal{N}^i(ds,du,dh).
\end{multline*}
Therefore, summing on all $i=1,\dots, N$ and taking expectation, by the application of Fatou's Lemma we get:

\begin{multline*}
\Mean\left[\frac{1}{N}\sum_{i=1}^N\|X^N_i(t)\|\right]\leq \liminf_{\epsilon\downarrow 0}\left(\Mean\left[\frac{1}{N}\sum_{i=1}^Nf^{\epsilon}(X^N_i(0))\right]\right.\\
\left.+\int_0^t\left(\Mean[f^{\epsilon}(V)]+\Mean[f^{\epsilon}(U)]\right)\Mean\left[\frac{1}{N}\sum_{i=1}^N\lambda(X^N_i(s))\right]-\Mean\left[\frac{1}{N}\sum_{i=1}^Nf^{\epsilon}(X^N_i(s))\lambda(X^N_i(s))\right]ds\right)\\
\end{multline*}
Then, by monotone convergence, we have 
\begin{multline*}
\Mean\left[\frac{1}{N}\sum_{i=1}^N\|X^N_i(t)\|\right]\leq\Mean\left[\frac{1}{N}\sum_{i=1}^N\|X^N_i(0)\|\right]\\
+\int_0^t\left(\Mean[\|V\|]+\Mean[\|U\|]\right)\Mean\left[\frac{1}{N}\sum_{i=1}^N\lambda(X^N_i(s))\right]
-\Mean\left[\frac{1}{N}\sum_{i=1}^N\|X^N_i(s)\|\lambda(X^N_i(s))\right]ds.
\end{multline*}
Since $b$ is increasing and $h$ is bounded, there exists a positive constant $C$, depending on \\
{\small$\Mean\left[\frac{1}{N}\sum_{i=1}^N\|X^N_i(0)\|\right]$}, such that
$$
\sup_{t\geq0}\Mean\left[\frac{1}{N}\sum_{i=1}^N\|X^N_i(t)\|\right]\leq C,$$
implying $\Prb(\tau>T)=1$.
\end{proof}
The proof of existence and uniqueness of solutions of \eqref{SDE_LIMITE_NONSTANDARD_RATE} for compact support initial condition relies on a straightforward adaptation of the arguments of \cite{RoTo14} to our framework, therefore we write the proof of Lemma~\ref{Lemma_Z} only for completeness.

\begin{proof}[Proof of Lemma \ref{Lemma_Z}.]
 We want to get an almost sure bound for $\|Z_f(t)\|$, in order to use locally Lipschitzianity of $\lambda$ in the following computations. Intuitively, the jumps have  an increasing role only if we are inside the support of the random variable $U$, otherwise they force the norm to decrease. Therefore, a.s., we can bound the process $\|Z_f(t)\|$ with the deterministic expression 
\[K_0+\int_0^t\|f(s)\|ds,
\]
where $K_0\colon= \max\{\|x\|,\sup_{h\in[0,1]}\|U(h)\|\}$. This almost sure bound for $\|Z_f(t)\|$ and the continuity of the coefficients ensure the existence and uniqueness of a non-explosive solution $Z_f$ on $[0,T]$. Let $Z_f$ and $Z_g$ two solutions of \eqref{SDE_con_funct} corresponding to two different locally bounded Borellian functions $f$ and $g$. The almost sure bounds on $\|Z_f(t)\|$ and $\|Z_g(t)\|$ let us define two positive constant $b_{f,g}(T)$ and $L_{f,g}(T)$, such that we have {\small
\begin{multline*}
\displaystyle{\Mean\left[\sup_{t\in[0,T]}\left\|Z_f(t)-Z_g(t)\right\|\right]\leq\int_0^T\Mean\left[\sup_{s\in[0,t]}\left\|Z_f(s)-Z_g(s)\right\|\right]ds+\int_0^T\sup_{s\in[0,t]}\|f(s)-g(s)\|dt}\\
\displaystyle{+\Mean\left[\int_0^T\int_{[0,1]\times[0,\infty)}\left\|(Z_f(s^-)-U(h))\mathds{1}_{[0,\lambda(Z_f(s^-)))}(u) -(Z_g(s^-)-U(h))\mathds{1}_{[0,\lambda(Z_g(s^-)))}(u) \right\|dsdu\nu_1(dh)\right]}\\
\displaystyle{\qquad\qquad\leq \int_0^T\Mean\left[\sup_{s\in[0,t]}\left\|Z_f(s)-Z_g(s)\right\|\right]ds+\int_0^T\sup_{s\in[0,t]}\|f(s)-g(s)\|dt}\\
\displaystyle{+\left( b_{f,g}(T)+H\right)\int_0^T\Mean\left[\sup_{s\in[0,t]}\left\|Z_f(s)-Z_g(s)\right\|\right]ds}\\
\displaystyle{+L_{f,g}(T)\left(\sup_{t\in [0,T]}\|Z_f(t)\|\right)\int_0^T\Mean\left[\sup_{s\in[0,t]}\left\|Z_f(s)-Z_g(s)\right\|\right]ds.}
\end{multline*}}
We apply now Gronwall lemma and we obtain \eqref{bound_Z_u}.
\end{proof}

The proof of Lemma~\ref{bound_a_priori} requires two technical lemmas adapted to our case from \cite{RoTo14}.

\begin{lemma}\label{limitatezza}
Let $x(t)$ be a non-negative $C^1$ function on $\mathbb{R}_+$. If the following inequality holds for any $0\leq s\leq t$:
$$
x(t) \leq x(s)-\bar{K}\int_s^tx^k(u)du+\int_s^tP_{\delta}\left(x(u)\right)du
$$
where $k,\bar K>0$ and $P_{\delta}(\cdot)$ is a polynomial of degree $\delta<k$, then 
$$
\sup_{t\geq 0}x(t)\leq C_0<\infty.$$
\end{lemma}
\begin{proof}
Consider that for $x\rightarrow\infty$, then 
$$
-\bar{K}x^k+P_{\delta}(x)\rightarrow- \infty.$$
Therefore it exists a value $\bar{C}_0$ such that, as soon as the trajectory exceeds $\bar{C}_0\geq0$ its derivative becomes strictly negative and the trajectory is forced toward zero. Thus, defining
$$C_0:=\max\{\bar{C}_0,x(0)\},$$
we get the desired bound.
\end{proof}
\begin{lemma}\label{lemma_b^p}
If the function $b$ satisfies the assumption~\eqref{condizione_su_b},  then  for any $\epsilon>0$ and $p$ $\in$ $[1,4+2\epsilon]$, there exists a constant $\gamma_1<(4+2\epsilon)\gamma$, $c_1>0$ and a value $\eta>0$, such that, for all $a\in \mathbb{R}^d$ with $\|a\|\leq \eta$ and for all $x\in\mathbb{R}^d$, it holds
\begin{equation*}
\left|b^p(\|x+a\|)-b^p(\|a\|)\right|\leq \|a\|\left(\gamma_1 b^p(\|x\|)+c_1\right). 
\end{equation*}
\end{lemma}
\begin{proof}
The proof of this lemma comes directly from Lemma~8 in the appendix of \cite{RoTo14}.
\end{proof}
Notice that the constant $\gamma_1<(4+2\epsilon)\gamma$, together with the condition of Lemma~\ref{limitatezza} on the negativity of the coefficient $\bar{K}$, cause the condition on $\gamma$ w.r.t $\Mean[\|V\|]$ in Assumption~\ref{ass}. This condition plays a crucial role in all the proofs of the boundedness for the moments of $\lambda(X(t))$ and of $\lambda(X^N_i(t))$ for all $i$. 
Now that we have stated these two results, we are ready to prove Lemma~\ref{bound_a_priori}, that provides a priori uniform bounds on the first moment of the solution to \eqref{SDE_LIMITE_NONSTANDARD_RATE} and on the moments of $\lambda(X(t))$.

\begin{proof}[Proof of Lemma \ref{bound_a_priori}]
{ Fix $\epsilon>0$, by means of Ito's rule, we have
\begin{multline*}
\Mean\left[f^{\epsilon}(X(t))\right]
\leq \Mean\left[f^{\epsilon}(X(0))\right]-\int_0^t\Mean\left[\|X(s)\|\mathds{1}(\|X(s)\|>\epsilon)\right]ds\\
-\int_0^t\Mean\left[\epsilon\mathds{1}(\|X(s)\|\leq\epsilon)\right]ds+\int_0^t\Mean\left[(\Mean\left[\|V\|\right]+\Mean\left[\|U\|\right]-f^{\epsilon}(X(s)))h(X(s))\right]ds\\
+\int_0^t\Mean\left[b(\|X(s)\|)\left(\Mean\left[\|V\|\right]+\Mean\left[\|U\|\right]-f^{\epsilon}(X(s))\right)\right]ds.
\end{multline*}
For the monoticity assumption on $b$, we know that there exist $\Lambda>0$ and $\beta\geq0$ such that $b(r)\left(\Mean\left[\|V\|\right]+\Mean\left[\|U\|\right]-r\right)\leq-\Lambda r+\beta$. Therefore, by 
Fatou's lemma and monotone convergence theorem, 
\begin{equation*}
\Mean\left[\|X(t)\|\right]\leq \Mean\left[\|X(0)\|\right]+\int_0^t \left[H\left(\Mean\left[\|V\|\right]+\Mean\left[\|U\|\right]\right)+\beta\right] ds
-\Lambda\int_0^t\Mean\left[\|X(s)\|)\right]ds,
\end{equation*}
that gives the boundedness of $\sup_{t\geq0}\Mean[\|X(t)\|]$. \\
Let $p=1$, clearly, to get a bound for $\Mean[\lambda(X(t))]$, it is sufficient to bound $\Mean\left[b(\|X(t)\|)\right]$. Thus, again, we use Ito's rule to compute $b(f^{\epsilon}(X(t)))$ for $\epsilon>0$.
{\small
\begin{multline*}
\Mean\left[b(f^{\epsilon}(X(t)))\right]\leq\Mean\left[b(f^{\epsilon}(X(0)))\right]\\
\displaystyle{-\int_0^t\Mean\left[b'(f^{\epsilon}(X(s)))\|X(s)\|\mathds{1}(\|X(s)\|>\epsilon)\right]ds-\int_0^t\Mean\left[b'(f^{\epsilon}(X(s)))\frac{\|X(s)\|^2}{\epsilon}\mathds{1}(\|X(s)\|\leq\epsilon)\right]ds}\\
\displaystyle{+\int_0^t\Mean\left[b'(f^{\epsilon}(X(s)))\Mean\left[b(\|X(s)\|)\right]\frac{X(s)\cdot \Mean[V]}{\|X(s)\|}\mathds{1}(\|X(s)\|>\epsilon)\right]ds}\\
\displaystyle{+H\int_0^t\Mean\left[b'(f^{\epsilon}(X(s)))\frac{X(s)\cdot \Mean[V]}{\|X(s)\|}\mathds{1}(\|X(s)\|>\epsilon)\right]ds}\\
+\int_0^t\Mean\left[b'(f^{\epsilon}(X(s)))\Mean\left[b(\|X(s)\|)\right]\frac{X(s)\cdot \Mean[V]}{\epsilon}\mathds{1}(\|X(s)\|\leq\epsilon)\right]ds\\
+H\int_0^t\Mean\left[b'(f^{\epsilon}(X(s)))\frac{X(s)\cdot \Mean[V]}{\epsilon}\mathds{1}(\|X(s)\|\leq\epsilon)\right]ds\\
+\int_0^t\Mean\left[b(\|X(s)\|)\right]\Mean\left[b(f^{\epsilon}(U))\right]ds+\int_0^t\Mean\left[h(X(s))\right]\Mean\left[b(f^{\epsilon}(U))\right]ds\\
-\int_0^t\Mean\left[b(f^{\epsilon}(X(s)))b(\|X(s)\|)\right]ds-\int_0^t\Mean\left[b(f^{\epsilon}(X(s)))\right]\Mean\left[h(X(s))\right]ds.
\end{multline*}}
Again we use Fatou's lemma and monotone convergence theorem (indeed $b(f^{\epsilon}(\cdot))$ converges monotonically to $b(\|\cdot\|)$, thanks to the increasing property of $b$). Since $b'$ is positive, we disregard the two terms in the second row, we use properties of $b'$ to bound the remaining terms and we get
{\small\begin{multline*}
\Mean\left[b(\|X(t)\|)\right]
\leq \Mean\left[b(\|X(0)\|)\right]+\left(Hc\Mean\left[\|V\|\right]+H\Mean\left[b(\|U\|)\right]\right)t+\left(\gamma\Mean\left[\|V\|\right]-1\right)\int_0^t\Mean\left[b(\|X(s)\|)\right]^2ds\\
\left(c\Mean\left[\|V\|\right]+H\gamma\Mean\left[\|V\|\right]+\Mean\left[b(\|U\|)\right]+H\right)\int_0^t\Mean\left[b(\|X(s)\|)\right]ds.
\end{multline*}}
With Lemma~\ref{limitatezza} we conclude the boundedness for $\Mean\left[b(\|X(t)\|)\right]$. The same argument is used to get a uniform bound for $\Mean\left[b^p(\|X(t)\|)\right]$ when $p=2,3,4$.}
\end{proof}
While the uniform bounds for $\Mean\left[\|X(t)\|\right]$ and $\Mean\left[b(\|X(t)\|)\right]$ are needed for the well-posedness of the nonlinear process itself, higher moments of $\lambda$ are needed only for the proof of propagation of chaos. The same a priori bounds for the moments of $\lambda$ appear also in the case of the particle system. Their proof is similar to the nonlinear case, relies on Lemma~\ref{limitatezza} and Lemma~\ref{lemma_b^p}, together with an argument based on orthogonal martingales.

\begin{proof}[Proof of Lemma \ref{bound_per_<mu,f^5>}.]
We only prove it for $\mu^N_X$, then for $\mu^N_Y$ the steps are basically the same. Of course it is sufficient to prove the boundedness of $\sup_{N\geq N_0}\sup_{t\geq0}\Mean[\langle\mu^N_X(t),b^4(f^{\delta}(\cdot))\rangle]$. Let us define for $K>0$ the stopping time $\tau_K:=\inf\left\{t\geq 0:\langle \mu^N_X(t),b^5(f^{\delta}(\cdot))\rangle \geq K \right\}$. Obviously the random variables  $\langle \mu^N_X(t \wedge \tau_K), b^p(f^{\delta}(\cdot))\rangle $  for $1\leq p\leq 5$ and  $\langle \mu^N_X(t \wedge \tau_K), f^{\delta}(\cdot)\rangle $  are integrable. Recall that, for all $\epsilon>0$, the process $\{M^N_{\epsilon}(t)\}_{t\geq[0,T]}$, where, for $t\in[0,T]$ we have 
\begin{multline*}
M^N_{\epsilon}(t)\doteq \langle \mu^N_X(t),f^{\epsilon}(\cdot)\rangle-\langle \mu^N_X(0),f^{\epsilon}(\cdot)\rangle\\
+\frac{1}{N}\sum_{i=1}^N\left(\int_0^t\frac{X^N_i(s)\cdot X^N_i(s)}{\|X^N_i(s)\|}\mathds{1}(\|X^N_i(s)\|>\epsilon)ds+\int_0^t\frac{X^N_i(s)\cdot X^N_i(s)}{\epsilon}\mathds{1}(\|X^N_i(s)\|\leq\epsilon)ds\right)\\
-\frac{1}{N}\sum_{i=1}^N\sum_{j\neq i} \int_0^t\int_{[0,1]^{\mathbb{N}}}\lambda(X^N_j(s))\left(f^{\epsilon}\left(X^N_i(s)+\frac{V(h_i,h_j)}{N}\right)-f^{\epsilon}\left(X^N_i(s)\right)\right)\nu(dh)ds\\
-\frac{1}{N}\sum_{i=1}^N\int_{0}^t\int_{[0,1]}^{\mathbb{N}}\lambda(X^N_i(s))\left(f^{\epsilon}\left(U(h_i)\right)-f^{\epsilon}\left(X^N_i(s)\right)\right)\nu(dh)ds,
\end{multline*}
is a martingale. Then, for the optional stopping theorem, it holds
\begin{multline*}
\Mean \left[\langle\mu^N_X(t\wedge\tau_K),f^{\epsilon}(\cdot)\rangle\right]\leq \Mean \left[\mu^N_X(0),f^{\epsilon}(\cdot)\rangle\right]
-\Mean \left[\int_0^{t\wedge\tau_K}\langle\mu^N_X(s),\|\cdot\|\mathds{1}(\|\cdot\|>\epsilon)\rangle ds\right]\\
-\Mean \left[\int_0^{t\wedge\tau_K}\langle\mu^N_X(s),\frac{\|\cdot\|^2}{\epsilon}\mathds{1}(\|\cdot\|\leq\epsilon)\rangle ds\right]\\
+N\Mean \left[\int_0^{t\wedge\tau_K}\langle\mu^N_X(s),\lambda(\cdot)\rangle\langle\mu^N_X(s),\int_{[0,1]^{2}}f^{\epsilon}\left(\cdot+\frac{V(h_1,h_2)}{N}\right)-f^{\epsilon}(\cdot)\nu_2(dh)\rangle ds\right]\\
-\Mean \left[\int_0^{t\wedge\tau_K}\langle\mu^N_X(s),\lambda(\cdot)\int_{[0,1]}f^{\epsilon}\left(\cdot+\frac{V(h_1,h_1)}{N}\right)-f^{\epsilon}(\cdot)\nu_1(dh)\rangle ds\right]\\
+\Mean \left[\int_0^{t\wedge\tau_K} \Mean[f^{\epsilon}(U)]\langle\mu^N_X(s),\lambda(\cdot)\rangle-\langle\mu^N_X(s),\lambda(\cdot)f^{\epsilon}(\cdot)\rangle ds\right].
\end{multline*}
Again, we use the monotone convergence of $f^{\epsilon}(x)$ to $\|x\|$, to get
{\small\begin{equation*}
\Mean \left[\mathds{1}(t\leq \tau_K)\langle\mu^N_X(t),\|\cdot\|\rangle\right]\leq \liminf_{\epsilon\downarrow 0}\Mean \left[\mathds{1}(t\leq \tau_K)\langle\mu^N_X(t),f^{\epsilon}(\cdot)\rangle\right]\leq\liminf_{\epsilon\downarrow 0}\Mean \left[\langle\mu^N_X(t\wedge\tau_K),f^{\epsilon}(\cdot)\rangle\right].
\end{equation*}}
By arguments close to the one in the proof of Lemma~\ref{bound_a_priori}, there exists $\Lambda>0$ and $\beta>0$, such that we get the following inequality
\begin{multline*}
\Mean \left[\mathds{1}(t\leq \tau_K)\langle\mu^N_X(t),\|\cdot\|\rangle\right]\leq \Mean \left[\mu^N_X(0),\|\cdot\|\rangle\right]\\
+\int_0^{t}\Mean\left[\mathds{1}(s\leq\tau_K)\langle\mu^N_X(s),\left(\Mean[\|V\|]+\frac{\Mean[\|V\|]}{N}+\Mean[\|U\|]-\|\cdot\|\right)\lambda(\cdot)\rangle\right]ds\\
\leq \Mean \left[\mu^N_X(0),\|\cdot\|\rangle\right]+\left[H\left(\Mean[\|V\|]+\frac{\Mean[\|V\|]}{N}+\Mean[\|U\|]\right)+\beta\right] t \\
-\Lambda\int_0^{t}\Mean\left[\mathds{1}(s\leq\tau_K)\langle\mu^N_X(s),\|\cdot\|\rangle\right]ds.
\end{multline*}
This, together with Lemma~\ref{limitatezza}, gives the boundedness of $\sup_{t\geq0}\Mean \left[\mathds{1}(t\leq\tau_K)\langle\mu^N_X(t),\|\cdot\|\rangle\right]$. Since this bound does not depend on $K$, we are allow to let $K$ go to infinity, and therefore obtain a bound on 
$\displaystyle{\sup_{t\geq0}\Mean \left[\langle\mu^N_X(t),\|\cdot\|\rangle\right]}$. Now we apply the same argument to the martingale $\{M^N_{b^4}(t)\}_{t\geq[0,T]}$. By deleting some of the negative terms, applying Lemma~\ref{lemma_b^p} and repeating the previous steps, we obtain the following bound
{\small
\begin{multline*}
\displaystyle{\Mean\left[\mathds{1}(\tau_K\leq t)\langle \mu^N_X(t),b^4(\|\cdot\|)\rangle\right]}
%
%
%
%
\leq   \Mean\left[ \langle\mu^N_X(0), b^4(\|\cdot\|)\rangle\right]\\
+\gamma_1\Mean\left[\|V\|\right]\int_0^t\Mean\left[\mathds{1}(s\leq\tau_K)\langle\mu^N_X(s), b^4(\|\cdot\|)\rangle\langle\mu^N_X(s),b(\|\cdot\|)\rangle\right]ds\\
+H\gamma_1\Mean\left[\|V\|\right]\int_0^t\Mean\left[\mathds{1}(s\leq\tau_K)\langle\mu^N_X(s), b^4\left(\|\cdot\|\right)\rangle\right]ds\\
+c_1\Mean\left[\|V\|\right]\int_0^t\Mean\left[\mathds{1}(s\leq\tau_K)\langle\mu^N_X(s),b(\|\cdot\|)\rangle\right]ds+c_1H\Mean\left[\|V\|\right]\int_0^t\Mean\left[\mathds{1}(s\leq\tau_K)\right]ds\\
+\gamma_1\frac{\Mean\left[\|V\|\right]}{N}\int_0^t\Mean\left[\mathds{1}(s\leq\tau_K)\langle\mu^N_X(s), b^5(\|\cdot\|)\rangle\right]ds+c_1\frac{\Mean\left[\|V\|\right]}{N}\int_0^t\Mean\left[\mathds{1}(s\leq\tau_K)\langle\mu^N_X(s), b(\|\cdot\|)\rangle\right]ds\\
+H\gamma_1\frac{\Mean\left[\|V\|\right]}{N}\int_0^t\Mean\left[\mathds{1}(s\leq\tau_K)\langle\mu^N_X(s), b^4(\|\cdot\|)\rangle\right]ds+c_1H\frac{\Mean\left[\|V\|\right]}{N}\int_0^t\Mean\left[\mathds{1}(s\leq\tau_K)\right]ds\\
+\Mean\left[b^4(\|U\|)\right]\int_0^t\Mean\left[\mathds{1}(s\leq\tau_K)\langle\mu^N_X(s), b(\|\cdot\|)\rangle\right]ds+\Mean\left[b^4(\|U\|)\right]Ht\\
-\int_0^t\Mean\left[\mathds{1}(s\leq\tau_K)\langle\mu^N_X(s), b^5(\|\cdot\|)\rangle\right]ds+H\int_0^t\Mean\left[\mathds{1}(s\leq\tau_K)\langle\mu^N_X(s), b^4(\|\cdot\|)\rangle\right]ds.
\end{multline*} }
By H\"{o}lder and Jensen inequalities, we get the following expression
{\small\begin{multline*}
\displaystyle{\Mean\left[\langle \mathds{1}(\tau_K\leq t)\mu^N_X(t),b^4(\|\cdot\|)\rangle\right]\leq}  \Mean\left[ \langle\mu^N_X(0), b^4(\|\cdot\|)\rangle\right]\\
+\gamma_1\Mean\left[\|V\|\right]\int_0^t\Mean\left[\mathds{1}(s\leq\tau_K)\langle\mu^N_X(s), b^4(\|\cdot\|)\rangle\right]^{5/4}ds+H\gamma_1\Mean\left[\|V\|\right]\int_0^t\Mean\left[\mathds{1}(s\leq\tau_K)\langle\mu^N_X(s), b^4\left(\|\cdot\|\right)\rangle\right]ds\\
+c_1\Mean\left[\|V\|\right]\int_0^t\Mean\left[\mathds{1}(s\leq\tau_K)\langle\mu^N_X(s), b^4(\|\cdot\|)\rangle\right]^{1/4}ds+c_1H\Mean\left[\|V\|\right]t\\
+c_1\frac{\Mean\left[\|V\|\right]}{N}\int_0^t\Mean\left[\mathds{1}(s\leq\tau_K)\langle\mu^N_X(s), b^4(\|\cdot\|)\rangle\right]^{1/4}ds+H\gamma_1\frac{\Mean\left[\|V\|\right]}{N}\int_0^t\Mean\left[\mathds{1}(s\leq\tau_K)\langle\mu^N_X(s), b^4(\|\cdot\|)\rangle\right]ds\\
+\left(c_1H\frac{\Mean\left[\|V\|\right]}{N}+\Mean\left[b^4(\|U\|)\right]\right)t
+\Mean\left[b^4(\|U\|)\right]\int_0^t\Mean\left[\mathds{1}(s\leq\tau_K)\langle\mu^N_X(s), b^4(\|\cdot\|)\rangle^{1/4}\right]ds\\
+H\int_0^t\Mean\left[\mathds{1}(s\leq\tau_K)\langle\mu^N_X(s), b^4(\|\cdot\|)\rangle\right]ds+\left(\gamma_1\frac{\Mean\left[\|V\|\right]}{N}-1\right)\int_0^t\Mean\left[\mathds{1}(s\leq\tau_K)\langle\mu^N_X(s), b^4(\|\cdot\|)\rangle\right]^{5/4}ds,
\end{multline*} }
where we have exploited the fact that $\left(\gamma_1\frac{\Mean\left[\|V\|\right]}{N}-1\right)<0$, for $N$ large enough, and that \\
$\langle\mu^N_X(s),b^5\rangle\geq\langle\mu^N_X(s),b^4\rangle^{5/4}$.  Reordering, we get
{\small\begin{multline*}
\displaystyle{\Mean\left[\mathds{1}(\tau_K\leq t)\langle \mu^N_X(t),b^4(\|\cdot\|)\rangle\right]\leq} \langle \Mean\left[ \mu^N_X(0), b^4(\|\cdot\|)\rangle\right] \\
+\left(c_1\Mean[\|V\|]+c_1\frac{\Mean[\|V\|]}{N}+\Mean\left[b^4(\|U\|)\right]\right)\int_0^t\Mean\left[\mathds{1}(s\leq\tau_K)\langle\mu^N_X(s), b^4(\|\cdot\|)\rangle\right]^{1/4}ds\\
+\left(H\gamma_1\Mean[\|V\|]+H\gamma_1\frac{\Mean[\|V\|]}{N}+H\right)\int_0^t\Mean\left[\mathds{1}(s\leq\tau_K)\langle\mu^N_X(s), b^4(\|\cdot\|)\rangle\right]ds\\
+\left(\gamma_1\Mean\left[\|V\|\right]+\gamma_1\frac{\Mean\left[\|V\|\right]}{N}-1\right)\int_0^t\Mean\left[\mathds{1}(s\leq\tau_K)\langle\mu^N_X(s), b^4(\|\cdot\|)\rangle\right]^{5/4}ds.
\end{multline*} }
Since, by hypothesis, there exists $N_0$ such that, for all $N\geq N_0$ it holds $$\left(\gamma_1\Mean\left[\|V\|\right]+\gamma_1\frac{\Mean\left[\|V\|\right]}{N}-1\right)<0,$$ we use Proposition~\ref{limitatezza} and this gives a bound on
$\displaystyle{\Mean\left[\mathds{1}(t\leq \tau_K)\langle \mu^N_X(t),b^4(\|\cdot\|)\rangle\right]}$ independent of $N$ and $K$; therefore letting $K$ go to infinity proves the thesis.
\end{proof}

As mentioned before, Lemma~\ref{bound_per_<mu,f^5>} plays a crucial role in the proof of Lemma~\ref{bound_jumps}, where we bound the number of jumps of a single particle for the particle system~\eqref{SDE_XN_RATE_NONSTANDARD} and the contribution of the collateral drift term for the particle system~\eqref{SDE_YN_RATE_NONSTANDARD}.

\begin{proof}[Proof. of Lemma \ref{bound_jumps}.]
We develop the computations for the proof just in the case of \eqref{SDE_XN_RATE_NONSTANDARD}, since for the system~\eqref{SDE_YN_RATE_NONSTANDARD} they are almost the same. Let us start by describing the quantity $C_N(T)$, that is
$$
\displaystyle{C_N(T)=\sum_{i=1}^N\int_0^T\int_{[0,1]^{\mathbb{N}}}\int_0^{\infty}\mathds{1}_{[0,\lambda(X^N_i(s))}(u)\mathcal{N}^i(ds,du,dh)}.$$
We can rewrite this quantity as the sum of orthogonal martingales, that we will indicate as $M^N(t)$, plus a term depending on the empirical measure, as follows:
\begin{multline*}
\frac{C_N(T)}{N}=\frac{1}{N}\sum_{i=1}^N\int_0^T\int_{[0,1]^{\mathbb{N}}}\int_0^{\infty}\mathds{1}_{[0,\lambda(X^N_i)(s)}(u)\tilde{\mathcal{N}}^i(ds,du,dh)+\int_0^T\langle \mu_X^N(s),\lambda(\cdot)\rangle ds\\
\doteq M^N(T)+\int_0^T\langle \mu_X^N(s),\lambda(\cdot)\rangle ds.
\end{multline*}
Let us consider  a positive constant $H_T>0$, then 
$$
\Prb\left(\frac{C_N(T)}{N}\geq H_T\right)\leq \Prb\left(M^N(T)\geq H_T\right)+\Prb\left(\int_0^T\langle\mu^N_X(s),\lambda \rangle ds\geq H_T\right).
$$
Of course, since $\{M^N(t)\}_{t\in[0,T]}$ is a martingale, we have $\Prb\left(M^N(T)\geq H_T\right)\leq \frac{\Mean[M^N(T)]}{H_T}=0$. Therefore, we want to get a bound for the probability $\Prb\left(\int_0^T\langle\mu^N_X(s),\lambda \rangle ds\geq H_T\right)$.  Let $\delta>0$ be fixed, the first step consists in proving that  there exists $C_T>0$ such that
\begin{equation*}
\Mean\left[\sup_{t\in[0,T]}M^N_{b,\delta}(t)^2\right]\leq \Mean\left[\langle M^N_{b,\delta}(T)\rangle\right]\leq\frac{C_T}{N},
\end{equation*}
where 
$\{M^N_{b,\delta}(t)\}_{t\in[0,T]}$ is the martingale arising from the compensated Poisson measure in the computation of $\langle \mu^N_X(t), b(f^{\delta}(\cdot))\rangle$ with Ito rule, that is 
\begin{multline*}
M^N_{b,\delta}(t)\doteq \frac{1}{N}\sum_{i=1}^N\int_0^t\int_{[0,1]^{\mathbb{N}}}\int_0^{\infty}\mathds{1}_{(0,\lambda(X^N_i(s))]}\left[b(f^{\delta}(U(h_i)))-b(f^{\delta}(X^N_i(s)))\right.\\
\left.+\sum_{j\neq i}b\left(f^{\delta}\left(X^N_j(s)+\frac{V(h_i,h_j)}{N}\right)\right)-b\left(f^{\delta}\left(X^N_j(s)\right)\right)\right]\tilde{\mathcal{N}}^i(ds,du,dh)
\end{multline*} and $\langle M^N_{b,\delta}(t)\rangle$ is its quadratic variation. We use the fact that  $\{\tilde{\mathcal{N}}^i\}_{i=1,2\dots}$ is a family of orthogonal martingales, therefore
\begin{multline*}\label{variaz_quadratica_M_b}
\langle M^N_{b,\delta}(t)\rangle= \frac{1}{N^2}\sum_{i=1}^N\int_0^t\int_{[0,1]^{\mathbb{N}}}\lambda(X^N_i(s))\left[b(f^{\delta}(U(h_i)))-b(f^{\delta}(X^N_i(s)))\right.\\
\left.+\sum_{j\neq i}b\left(f^{\delta}\left(X^N_j(s)+\frac{V(h_i,h_j)}{N}\right)\right)-b\left(f^{\delta}\left(X^N_j(s)\right)\right)\right]^2\nu(dh)ds.
\end{multline*}

Let us write $\langle M^N_{b,\delta}(t)\rangle\doteq \frac{1}{N^2}\sum_{i=1}^N M^N_{b,\delta,i}(t)$, we fix $i$ and we compute $M_{b,\delta,i}(t)$ as follows.
{\small 
\begin{multline*}
M_{b,\delta,i}(t)\leq2\int_0^t\int_{[0,1]^{\mathbb{N}}}b(f^{\delta}(X^N_i(s)))b^2(f^{\delta}(U))+Hb^2(f^{\delta}(U))+b^3(f^{\delta}(X^N_i(s)))+Hb^2(f^{\delta}(X^N_i(s)))\nu(dh)ds\\
+\int_0^t\int_{[0,1]^{\mathbb{N}}}b(f^{\delta}(X^N_i(s)))(N-1)\sum_{j\neq i}\left(\frac{f^{\delta}(V)}{N}(\gamma_1b(f^{\delta}(X^N_j(s)))+c_1)\right)^2\nu(dh)ds\\
+\int_0^t\int_{[0,1]^{\mathbb{N}}}H(N-1)\sum_{j\neq i}\left(\frac{f^{\delta}(V)}{N}(\gamma_1b(f^{\delta}(X^N_j(s)))+c_1)\right)^2\nu(dh)ds\\
+2\int_0^t\int_{[0,1]^{\mathbb{N}}}b(f^{\delta}(X^N_i(s)))(b(f^{\delta}(U))-b(f^{\delta}(X^N_i(s))))\sum_{j\neq i}\left(\frac{f^{\delta}(V)}{N}(\gamma_1b(f^{\delta}(X^N_j(s)))+c_1)\right)\nu(dh)ds\\
+2H\int_0^t\int_{[0,1]^{\mathbb{N}}}(b(f^{\delta}(U))-b(f^{\delta}(X^N_i(s))))\sum_{j\neq i}\left(\frac{f^{\delta}(V)}{N}(\gamma_1b(f^{\delta}(X^N_j(s)))+c_1)\right)\nu(dh)ds\\
\leq \left(2H\Mean[b^2(f^{\delta}(U))]+Hc_1^2\Mean[f^{\delta}(V)^2]\frac{N-1}{N}+2c_1\Mean[b(f^{\delta}(U))]\Mean[f^{\delta}(U)]H\right)t\\
+\left(2\Mean[b^2(f^{\delta}(U))]+c_1^2\Mean[f^{\delta}(V)^2]\frac{N-1}{N}+2c_1\Mean[b(f^{\delta}(U))]\Mean[f^{\delta}(V)]+2c_1\Mean[f^{\delta}(V)]H\right)\int_0^tb(f^{\delta}(X^N_i(s)))ds\\
+\left(2H+2c_1\Mean[f^{\delta}(V)]\right)\int_0^tb^2(f^{\delta}(X^N_i(s)))ds+\int_0^tb^3(f^{\delta}(X^N_i(s)))ds\\
+\left(2\gamma_1\Mean[b(f^{\delta}(U))]\Mean[f^{\delta}(V)]+2\gamma_1\Mean[f^{\delta}(V)]H\right)\int_0^tb(f^{\delta}(X^N_i(s)))\langle\mu^N_X(s),b(f^{\delta}(\cdot))\rangle ds\\
+\gamma_1^2\Mean[f^{\delta}(V)^2]\frac{N-1}{N}\int_0^tb(f^{\delta}(X^N_i(s)))\langle\mu^N_X(s),b^2(f^{\delta}(\cdot))\rangle ds\\
+H\gamma_1\Mean[f^{\delta}(V)^2]\frac{N-1}{N}\int_0^t\langle\mu^N_X(s),b^2(f^{\delta}(\cdot))\rangle ds+2\gamma_1\Mean[b(f^{\delta}(U))]\Mean[f^{\delta}(V)]\int_0^t\langle\mu^N_X(s),b(f^{\delta}(\cdot))\rangle ds\\
+2\gamma_1\Mean[f^{\delta}(V)]\int_0^tb^2(f^{\delta}(X^N_i(s)))\langle\mu^N_X(s),b(f^{\delta}(\cdot))\rangle ds
\end{multline*}}
Summing over all $i=1,\dots,N$ and dividing by $N^2$, we can find four positive constants $K_1,$ $K_2$, $K_3$ and $K_4$ such that  $\langle M^N_b(t)\rangle$ is bounded by the expression 
\begin{multline*}
\frac{K_1}{N}t+\frac{K_2}{N}\int_0^t\langle\mu^N_X(s),b^3(f^{\delta}(\cdot))\rangle^{1/3} ds+\frac{K_3}{N}\int_0^t\langle\mu^N_X(s),b^3(f^{\delta}(\cdot))\rangle^{2/3} ds+\frac{K_4}{N}\int_0^t\langle\mu^N_X(s),b^3(f^{\delta}(\cdot))\rangle ds
\end{multline*}
Using the result of Lemma~\ref{bound_per_<mu,f^5>}, we know that there exists a certain $N_0$, such that the expectation of all the terms involved is bounded uniformly in $N>N_0$. Therefore, for such $N$ we have
\begin{equation*}
\displaystyle{\Mean\left[\sup_{t\in[0,T]}M^N_{b,\delta}(t)\right]\leq \frac{C_T}{N}.}
\end{equation*} 
By Chebychev and Doob inequalities this leads to 
\begin{equation*}\label{bound_in_prob_per_M^N_b}
\displaystyle{\Prb\left(\sup_{t\in[0,T]}M^N_{b,\delta}(t)\geq 1\right)\leq \Mean\left[\sup_{t\in[0,T]}\left(M^N_{b,\delta}(t)\right)^2\right]\leq\Mean\left[\langle M^N_{b,\delta}(T)\rangle\right]\leq\frac{C_T}{N}}.
\end{equation*}
Now, we compute $\langle\mu^N_X(t),b(f^{\delta}(\cdot))\rangle$ with Ito's rule, that gives the following bound:
\small\begin{multline*}
\displaystyle{\langle\mu^N_X(t),b(f^{\delta}(\cdot))\rangle\leq \langle\mu^N_X(0),b(f^{\delta}(\cdot))\rangle+M^N_{b,\delta}(t)+\left(\Mean[f^{\delta}(V)]\gamma_1\left(1+\frac{1}{N}\right)-1\right)\int_0^t\langle\mu^N_X(t),b^2(f^{\delta}(\cdot))\rangle ds}\\
\displaystyle{+\left(H\Mean[f^{\delta}(V)]\gamma_1\left(1+\frac{1}{N}\right)+H+\Mean[f^{\delta}(V)]c_1\left(1+\frac{1}{N}\right)+\Mean[b^2(f^{\delta}(U))]\right)\int_0^t\langle\mu^N_X(t),b^2(f^{\delta}(\cdot))\rangle^{1/2} ds}\\
\displaystyle{+H\left(\Mean[f^{\delta}(V)]c_1+\Mean[b^2(f^{\delta}(U))]\right)t.}
\end{multline*}
Since, for hypothesis, $b(f^{\delta}(\cdot))$ is integrable with respect to the law of $X(0)$, for the law of large number, we know that 
\begin{equation*}
\Prb\left(\langle\mu^N_X(0),b(f^{\delta}(\cdot))\rangle \geq 1+\Mean[b(f^{\delta}(X(0)))]\right)\leq \frac{Var\left(b(f^{\delta}(X(0)))\right)}{N}.
\end{equation*}
Let us consider the event 
\begin{equation*}
\displaystyle{\left\{\langle\mu^N_X(0),b(f^{\delta}(\cdot))\rangle < 1+ \Mean[b(f^{\delta}(X(0)))]\right\}\cup\left\{\sup_{t\in[0,T]}M^N_{b,\delta}(t)<1\right\}},
\end{equation*}
that has a probability greater than $1-2\frac{C}{N}$. Under this event, we apply Lemma~\ref{limitatezza} to get a bound for $\langle\mu^N_X(T),b(f^{\delta}(\cdot))\rangle$.

Since, for all $\delta>0$, $\lambda(\cdot)\leq b(f^{\delta}(\cdot))+H$ a.s.,  this is equivalent to a bound for $\sup_{t\in[0,T]}\langle\mu^N_X(t),\lambda(\cdot)\rangle$,  that leads to the existence of a positive constant $K_T$ such that
$$
\Prb\left(\int_0^T\langle\mu^N_X(s),\lambda \rangle ds\geq H_T\right)\leq \frac{K_T}{N},$$
and therefore to  the desired bound for $\Prb\left(\frac{C_N(T)}{N}\geq H_T\right)$. 

\end{proof}



\end{document}